\documentclass[twoside]{article}

\usepackage[accepted]{aistats2025}
\usepackage[round]{natbib}

\usepackage[utf8]{inputenc} % allow utf-8 input
\usepackage[T1]{fontenc}    % use 8-bit T1 fonts
\usepackage{hyperref}       % hyperlinks
\usepackage{url}            % simple URL typesetting
\usepackage{booktabs}       % professional-quality tables
\usepackage{amsfonts}       % blackboard math symbols
\usepackage{nicefrac}       % compact symbols for 1/2, etc.
\usepackage{microtype}      % microtypography
\usepackage{xcolor}         % colors

\usepackage{amsthm}
\usepackage{amsmath}
\usepackage{amssymb}
\usepackage{graphicx}
\usepackage{bbm}
\usepackage{braket,amsfonts,amssymb}
\usepackage[linesnumbered,ruled]{algorithm2e}
\usepackage{comment}
\usepackage{amsopn}
\usepackage{bm}
\usepackage{mathrsfs}  
\usepackage[mathscr]{euscript}
\usepackage{bm,nicefrac}
\usepackage{extarrows}
\usepackage{caption}
\usepackage{subcaption}
\usepackage{xcolor}

\newtheorem{theorem}{Theorem}[section]

\newtheorem{lemma}[theorem]{Lemma}
\newtheorem{definition}[theorem]{Definition}
\newtheorem{remark}[theorem]{Remark}

\newcommand\y[1]{\mathcal{#1}} % math cursive
\newcommand{\onenorm}[1]{\Vert #1 \Vert_1}

\let\cedil\c

\def\b{\mathbf{b}}
\def\bo{{\bf 1}}
\def\c{\mathbf c}
\def\cgd{c}

\def\etagd{\eta}

\def\j{j}

\def\n{n}
\def\nablaRiemann{\widetilde{\nabla}}

\def\k{k}

\def\s{s}
\def\t{t}
\def\u{{\bf u}}
\def\v{{\bf v}}
\def\w{{\bf w}}
\def\x{{\bf x}}

\newcommand{\ZETA}{\boldsymbol{\zeta}}
\def\y{{\bf b}}
\def\z{{\bf z}}

\def\we{w}

\def\xe{x}
\def\xprod{\tilde{\bf x}}
\def\xprode{\tilde{\xe}}

\def\xprodinfty{\tilde{\x}_\infty}

\def\halfstep{\pmb{\xi}}

\def\A{{\bf A}}

\def\F{F}

\def\L{L}
\def\LossBasic{\ell}
\def\LossPoly{\mathcal{L}}
\def\M{M}
\def\N{N}
\def\PC{\mathcal{P}_{\mathcal C}}
\def\Q{\mathbf{Q}}
\def\R{\mathbb{R}}
\def\T{\text{T}}

\def\0{\boldsymbol{0}}
\def\1{\boldsymbol{1}}

\def\Lover{\mathcal{L}_{\text{over}}}

\DeclareMathOperator*{\argmin}{arg\,min}

\renewcommand{\L}{L}

\newcommand{\diag}{\mathrm{diag}}
\renewcommand{\epsilon}{\varepsilon}

\newcommand{\norm}[2]{\left\| #1 \right\|_{#2}}

\begin{document}

% If your paper is accepted and the title of your paper is very long,
% the style will print as headings an error message. Use the following
% command to supply a shorter title of your paper so that it can be
% used as headings.
%
\runningtitle{Get rid of your constraints and reparametrize}

% If your paper is accepted and the number of authors is large, the
% style will print as headings an error message. Use the following
% command to supply a shorter version of the authors names so that
% they can be used as headings (for example, use only the surnames)
%
\runningauthor{Chou, Maly, Verdun, Costa, and Mirandola}

\twocolumn[

\aistatstitle{\qquad Get rid of your constraints and reparametrize: \newline
A study in NNLS and implicit bias}

\aistatsauthor{ Hung-Hsu Chou\footnotemark \setcounter{footnote}{0} \And Johannes Maly\footnotemark \setcounter{footnote}{0} \And  Claudio Mayrink Verdun\footnotemark}

\aistatsaddress{ Department of Mathematics \\ University of Pittsburgh \\ USA  \And  Department of Mathematics, \\ LMU Munich, and Munich \\ Center for Machine Learning (MCML), \\ Germany \And John A. Paulson School of  \\ Engineering and Applied Sciences, \\ Harvard University \\ USA } 

\aistatsauthor{Bernardo Freitas Paulo da Costa \And Heudson Mirandola}

\aistatsaddress{ Escola de Matemática Aplicada \\ Funda\cedil cão Getulio Vargas \\ Brasil \And Instituto de Matemática \\ Universidade Federal do Rio de Janeiro \\ Brasil } 

]

\footnotetext{Equal contribution.}

\begin{abstract}
 Over the past years, there has been significant interest in understanding the implicit bias of gradient descent optimization and its connection to the generalization properties of overparametrized neural networks. Several works observed that when training linear diagonal networks on the square loss for regression tasks (which corresponds to overparametrized linear regression) gradient descent converges to special solutions, e.g., non-negative ones. We connect this observation to Riemannian optimization and view overparametrized GD with identical initialization as a Riemannian GD. We use this fact for solving non-negative least squares (NNLS), an important problem behind many techniques, e.g., non-negative matrix factorization. We show that gradient flow on the reparametrized objective converges globally to NNLS solutions, providing convergence rates also for its discretized counterpart. Unlike previous methods, we do not rely on the calculation of exponential maps or geodesics. We further show accelerated convergence using a second-order ODE, lending itself to accelerated descent methods. Finally, we establish the stability against negative perturbations and discuss generalization to other constrained optimization problems.
\end{abstract}

%%%%%%%%%%%%%%%%%%%%%%%%%%%
\section{INTRODUCTION}
\label{sec:Introduction}
%%%%%%%%%%%%%%%%%%%%%%%%%%%

In an attempt to understand the still mysterious generalization properties of unregularized
gradient-based learning procedures for deep neural networks \citep{belkin2019reconciling,zhang2021understanding}, a recent line of research has been focusing on examining the implicit bias of (stochastic) gradient descent (GD) on overparametrized models. The fundamental question is why the networks found by GD generalized so well despite being trained to a perfect interpolation regime. The mathematical analysis of {this} phenomenon was first considered in simplified problem instances like linear classification \citep{soudry2018implicit,ji2019implicit}, matrix factorization \citep{gunasekar2017implicit,arora2019implicit,chou2020implicit}, matrix recovery \citep{geyer2019implicit,stoger2021small}, two-layer networks in the mean-field regime \citep{chizat2020implicit} and linear regression \citep{vaskevicius2019implicit,zhao2019implicit,chou2021more}. 
{In such settings, GD and its underlying gradient flow provably converge to particular solutions, e.g., parsimoniously structured ones. Many of the theoretical contributions are restricted to deep linear networks} \citep{Bach2019implicit,woodworth2020kernel,pesme2021implicit} under the so-called \emph{equal} or \emph{balanced} initialization \citep{hardt2016identity,arora2018optimization,azulay2021implicit,elkabetz2021continuous,li2023simplex}. {In this case, the overparametrization is equivalent to a reparametrization of the GD dynamics; cf. Remark \ref{remark2.2}} below. Moreover, the study of overparametrized models led to an investigation of constrained optimization problems. In the same way that the bias of GD has been shown to induce parsimonious solutions under certain assumptions, it is natural to ask if, for certain reparametrizations, the limit point found via GD could belong to a (convex) set with a certain geometry. As is folklore in the literature, the Riemannian interpretation of such reparametrization allows us to naturally express the geometric structure of the problem, in case it has constraints \citep{martinez2023accelerated}. However, to our knowledge, even in the overparametrized linear regression case, these contributions have only analyzed the continuous dynamics \citep{li2022implicit} or second-order
optimality conditions of such overparametrized objective functions \citep{ding2023squared}.

{In this paper, we take inspiration from the rich body of work on the implicit bias of GD to analyze the performance of reparametrized GD on constrained optimization problems and propose accelerated GD variants for such problems. We use the so-called \emph{non-negative least-squares (NNLS)} problem as a popular representative. NNLS can be formulated in a rigorous way as follows: given} a linear operator $\A \in \R^{\M\times \N}$ and data ${\b} \in \R^\M$, one seeks for any solution
\begin{align} \label{eq:NNLS}
    \x_+ \in S_+ := \argmin_{\z \ge \0} \; {\LossBasic(\z)},
    \tag{NNLS}
\end{align} 
{where the least squares loss is defined as
\begin{equation}\label{eq:LossBasic}
    \LossBasic \colon \mathbb R^\N \to [0,\infty),
    \qquad
    \LossBasic(\z) := \frac{1}{2}\|\A\z -\b \|_2^2.
\end{equation} }
Since physical quantities of interest often are positive by nature, e.g., in deconvolution and demixing problems like source separation, it is common to search for a least-squared solution under additional non-negativity constraints by solving \eqref{eq:NNLS}. Large scale applications in which \eqref{eq:NNLS} appears include NMR relaxometry \citep{stoch2021enhanced}, imaging deblurring \citep{benvenuto2009nonnegative}, biometric templates  \citep{kho2019cancelable}, transcriptomics \citep{kim2011isoformex}, magnetic microscopy \citep{myre2019using}, sparse hyperspectral unmixing \citep{esser2013method}, and system identification \citep{chen2011nonnegative}. The formulation in \eqref{eq:NNLS} is also closely related to non-negative matrix factorization \citep{chu2021alternating} and supervised learning methods such as support vector machines \citep{vapnik1999nature}. We refer the interested reader to \cite{chen2010nonnegativity} for a survey about the development of algorithms that enforce non-negativity. 
In comparison with standard least squares, the positivity constraint, however, imposes additional difficulties. By now, three main algorithmic approaches exist: (i) interior point methods, (ii) active set methods, and (iii) projected gradient methods. In (i), one uses that \eqref{eq:NNLS} can be recast as the quadratic problem 
\begin{align} \label{eq:QuadraticFormulation}
    \argmin_{\x \geq \0} \frac{1}{2} \langle \x, \Q\x \rangle + \langle \c,\x\rangle,
\end{align}
where $\Q=\A^\top \A$ and $\c=-\A^\top \y$, and then be solved via interior point methods. For $\M=\N$, these are guaranteed to converge to an $\varepsilon$-solution in $\mathcal{O}(\N^3 \ln \varepsilon^{-1}) $ time \citep{bellavia2006interior}. The (ii) active set methods \citep{lawson1995solving} represent the most commonly used solution to \eqref{eq:NNLS}. They exploit the fact that $\x_+$ in \eqref{eq:NNLS} can be found by solving an unconstrained problem with \emph{inactive variables} that do not contribute to the constraints. Both (i) and (ii) require solving linear systems at each iteration, which limits their scalability. In contrast, (iii) projected gradient methods like projected gradient descent (PGD) only require matrix-vector multiplications, and the projection to the positive orthant can be trivially performed \citep[e.g.,][]{bierlaire1991iterative,lin2007projected, kim2013non}. Nevertheless, the choice of step size is challenging for such methods. Even though PGD is guaranteed to converge with the stepsize given by the inverse of the Lipschitz constant, this implies very slow convergence for ill-conditioned problems. When using stepsize acceleration methods in such scenarios, e.g., the Barzilai-Borwein step-size, one, however, encounters cases in which PGD exhibits cyclic behavior and fails to converge for a choice that provably works for unconstrained gradient descent
\citep{dai2005projected}. {For additional literature on implicit regularization of GD and NNLS we refer the reader to Appendix \ref{sec:RelatedWork}.}

{Instead of solving a constrained problem, we approach NNLS by considering GD and its underlying flow} on the \emph{reparametrized least squares loss} given by
\begin{equation}
\label{eq:LossPoly}
    \LossPoly \colon \mathbb R^\N \to [0,\infty),
    \qquad
    \LossPoly(\x) := \frac{1}{2\L}\|\A\x^{\odot L} - {\b} \|_2^2.
\end{equation}
where $\odot$ denotes the entry-wise (Hadamard) power; {the additional $\tfrac{1}{L}$-factor is added for convenience and does not influence the set of minimizers}. Therefore, we are going to analyze the flow $$\x'(t) = -\nabla \LossPoly(\x(\t)),$$ with $\x(0) = \x_0$ for a given $\x_0 \geq 0$, as well as its discrete counterpart, i.e., we will forget the non-negative constraints and analyze the behavior of the unconstrained non-convex flow. (In the following, we will often abbreviate $\xprod := \x^{\odot L}$ and use that $\LossPoly(\x) = \tfrac{1}{L} \LossBasic(\xprod)$.) Despite its simplicity, the {loss in \eqref{eq:LossPoly} --- also called \emph{Hadamard reparametrized least squares}} --- has been studied in various works \citep{arora2018optimization,gunasekar2019implicit,vaskevicius2019implicit,zhao2019implicit,chou2021more} because it is believed to be the first step in developing theoretical guarantees for models with deep structure.

\begin{remark}\label{remark2.2}
{The reparametrized loss in \eqref{eq:LossPoly} is related to the overparametrized loss $\Lover \colon \R^\N \times \cdots \times \R^\N \to [0,\infty)$ with}
\begin{align} \label{eq:Lover}
     \Lover\big(\x^{(1)}, \dots, \x^{(L)}\big)
    := \frac{1}{2}\,\Big\| \A \big( \x^{(1)} \odot \cdots \odot \x^{(L)} \big) - {\b}   \Big\|_2^2,
\end{align}
which appeared before in the context of implicit $\ell_1$-regularization of GD \citep{vaskevicius2019implicit,li2021implicit,chou2021more}.\footnote{{In fact, minimizing \eqref{eq:Lover} corresponds to training an $L$-layer linear network with diagonal weight matrices that are fully characterized by their diagonals $\x^{(\ell)} \in \R^\N$, cf.\ \cite{nacson2022implicit}.}}  In these works, it was shown that \underline{if there exists} a non-negative solution $\x \ge \0$ with $\A\x = \y$, and if all factors $\x^{(\k)}$ are initialized with $\alpha \bo$ at $\t=0$, for $\alpha > 0$ sufficiently small, {and follow gradient flow with respect to $\Lover$, then their product} $\x^{(1)}(\t) \odot \cdots \odot \x^{(L)}(\t)$ approximates an $\ell_1$-minimizer among all positive solutions, for $\t \to \infty$. 
A key observation by \cite{vaskevicius2019implicit,li2021implicit,chou2021more} is that if all factors $\x^{(k)}$ are initialized by the same vector $\x_0$, then gradient flow on \eqref{eq:Lover} will preserve this equality, and the trajectory of any $\x^{(k)}$ is equal to the trajectory of the gradient flow trajectory $\x$ on \eqref{eq:LossPoly} \cite[e.g.,][Appendix A.1]{chou2021more}.
\end{remark}

{In light of Remark \ref{remark2.2},} we will restrict ourselves to reparametrized gradient flow/descent on \eqref{eq:LossPoly}. Any resulting statement automatically applies to gradient flow/descent on \eqref{eq:Lover} under identical initialization of all factors $\x^{(k)}$.
{Let us highlight two points: first,} whereas in the works of \cite{vaskevicius2019implicit,li2021implicit,chou2021more} the positivity of the limit was a technical nuisance that needed to be circumvented by adapting \eqref{eq:Lover}, the present work makes use of it to restrict the flow to a constrained set and to solve \eqref{eq:NNLS}. 
Such a connection between overparameterization in ML and classical optimization problems opens up new avenues for algorithm design.
{Second, the fact that our results only apply to \eqref{eq:Lover} with} identical initialization, i.e., $\x^{(1)}(0) = \cdots = \x^{(L)}(0) = \x_0$ is not restrictive when solving \eqref{eq:NNLS}. In Appendix \ref{sec:ProofRemark1}, we argue that all solutions of \eqref{eq:NNLS} are stationary points of \eqref{eq:Lover} that can be described as the limit of gradient flow on \eqref{eq:Lover} under suitably chosen \emph{identical} initialization. We now present our main contributions, which demonstrate how constraints can be effectively traded for increased complexity in the objective function. This approach allows us to solve NNLS using unconstrained optimization techniques, leading to a scalable algorithm.
 
\subsection{Contribution and Outline} 

In light of the connection {between the implicit regularization of GD on \eqref{eq:LossPoly} and finding non-negative solutions in \eqref{eq:NNLS}, as discussed in Remark \ref{remark2.2}}, we analyze {in this work} the continuous and the discrete dynamics of the reparameterized loss $\LossPoly$ from equation~\eqref{eq:LossPoly} and show that:

\begin{itemize}
    \item The unconstrained gradient flow on \eqref{eq:LossPoly}
    globally converges to solutions of \eqref{eq:NNLS} at an $\mathcal{O}(1/\t)$ {decay rate on $\LossBasic$, cf.\ Theorem \ref{theorem:L1_equivalence_positive}. Whereas a similar result can be obtained via \cite{li2022implicit}, our presented proof does not directly rely on tools and concepts from Riemannian geometry, cf.\ Remark \ref{rem:Riemannian}.}.
    \item An accelerated version of the gradient flow on the original coordinates leads to an $\mathcal{O}(1/\t^2)$ {decay rate on $\LossBasic$, cf.\ Theorem \ref{thm:Accelerated}}. On numerical examples, this scheme attains an $\mathcal{O}(1/k^{3})$ decay. Surprisingly, accelerating the \emph{reparameterized gradient descent} (that is, for $\LossPoly(\x)$) yields even faster convergence, {cf.\ Section \ref{sec:accel}}.
    \item Gradient descent on \eqref{eq:LossPoly} {with decaying step-size schemes} exhibits {on $\LossBasic$ a decay rate of $\mathcal{O}(1/k^\gamma)$, where $0<\gamma < 1$ depends on the decay rate of the step-size, cf.\ Theorem \ref{thm:GD_ConvergenceRate_Improved}}. Such a result is stronger than the one given by mirror descent arguments on general convex loss function which, in turn, are similar to subgradient methods with decay rate $\mathcal{O}(1/\sqrt{k})$.\footnote{Note that mirror descent with 1-strongly convex Legendre function provably reaches $\mathcal{O}(1/k)$ \citep{lecture_notes}. However, this assumption is not fulfilled by the reparametrization in \eqref{eq:LossPoly}, which, to the best of our knowledge, is the only one considered in recent literature on Hadamard overparametrization.}
    Our numerical simulations suggest that our bound matches the practical performance of GD, {cf.\ Section \ref{sec:NumericsStepDecay}}.
\end{itemize}
A nice by-product of our approach is that, by choosing the initialization of gradient flow close to zero, one can add an additional $\ell_1$-bias on top of non-negativity. This latter feature is inherited from previous works like \cite{vaskevicius2019implicit,chou2021more} and is only of interest in the special case of applying NNLS to sparse recovery, {cf.\ Section \ref{sec:NNLSforSparseRecovery}}.

{The main body of the paper is structured as follows: In Section \ref{sec:MainResults}, we present our theoretical contributions in detail. In Section \ref{sec:Numerics}, we evaluate our methods in numerical experiments. We conclude in Section \ref{sec:Conclusion} with a discussion of the results.} All proofs and additional numerical experiments are provided in the appendix.

\subsection{Notation} 

Before detailing our results, we need to introduce some additional notation.
{For any solution $\x_+ \in S_+$ as defined in \eqref{eq:NNLS}, we define $\b_+ := \A\x_+$ which satisfies 
\begin{align} \label{eq:CPlus}
    \b_+ = \mathcal P_{C_+} \b,
    \text{ for }
    C_+ := \{ \A\z \colon \z \in \R_{\ge 0} \} \subset \R^\M.
\end{align}
Since $C_+$ is convex and closed, $\b_+$ is the unique Euclidean projection and thus independent of the choice of $\x_+$. Note that $\mathbb{R}_{\geq 0}^{\N} = \{\x\in\mathbb{R}^\N: \x \ge \0 \}$ and $\mathbb{R}_{> 0}^{\N} = \{\x\in\mathbb{R}^\N: \x > \0 \}$ where, for convenience, we denote by $\x \ge \bf{y}$ the entry-wise bound $\xe_n \ge y_n$, for all $n$, and the all zero and all ones vector by $\0$ and $\1$ respectively.
We denote $\ell_p$-vector norms by $\| \cdot \|_p$ and the operator norm of matrices by $\| \cdot \|$.
We use $\odot$ to denote the Hadamard product, i.e., the vectors $\x \odot \bf{y}$ and $\x^{\odot p}$ have entries $(\x\odot \mathbf{y})_n = x_n y_n$ and $(\x^{\odot p})_n = \xe_n^p$, respectively. Univariate functions like the logarithm are applied entry-wise to vectors, i.e., $\log(\x) \in \R^\N$ with $\log(\x)_n = \log(\xe_n)$. 
We denote the complement of a set $I$ by $I^c$ and its cardinality by $|I|$.} The support of a vector $\x \in \mathbb{R}^\N$ is the index set of its nonzero entries and denoted by $ \text{supp}(\x)=\{j \in [\N]: x_j \neq 0\}$. We call a vector $s$-sparse if $|\text{supp}(\x)| \le s$. We denote by $\x_I \in \R^\N$ the projection of $\x \in \R^\N$ onto the coordinates indexed by $I$. Furthermore, $\sigma_s(\x)_{\ell_1}$ denotes the $\ell_1$\emph{-error of  the best $s$-term approximation} of $\x \in \mathbb{R}^\N$, i.e., $\sigma_s(\x)_{\ell_1} = \inf \{\Vert \x-\z \Vert_1 \colon \z \in \mathbb{R}^\N \ \text{is $s$-sparse} \}$.

%%%%%%%%%%%%%%%%%%%%%%%%%%%
\section{THEORETICAL RESULTS}
\label{sec:MainResults}
%%%%%%%%%%%%%%%%%%%%%%%%%%%

{The main theoretical results of this paper show (i) that  reparametrized gradient flow on \eqref{eq:LossPoly} converges to solutions of \eqref{eq:NNLS} with decay rate $\mathcal{O}(1/t)$, see Section \ref{sec:ConvergenceGF}, (ii) that an accelerated version of the flow improves this rate to $\mathcal{O}(1/t^2)$, see Section \ref{sec:ConvergenceAcceleratedGF}, (iii) that GD on \eqref{eq:LossPoly} with decaying stepsize reaches a decay rate of $\mathcal{O}(1/k^\gamma)$ where $0<\gamma < 1$ depends on the decay rate of the step-size, see Section \ref{sec:ConvergenceGD}, and (iv) that our approach for solving NNLS comes with certain advantages compared to other NNLS methods when used for sparse recovery, see Section \ref{sec:NNLSforSparseRecovery}.}\footnote{{We emphasize at this point that, for a certain class of matrices $\A$, \eqref{eq:NNLS} is a valid approach to identifying sparse solutions from $\b$, cf.\ our extended discussion in Appendix \ref{sec:NNLS_sparse}.}}

%%%%%%%%%%%%%%%%%%%%%%%%%%%
\subsection{Convergence rate of gradient flow}
\label{sec:ConvergenceGF}
%%%%%%%%%%%%%%%%%%%%%%%%%%%

Our first contribution is to extend the argument of \cite{chou2021more} and to show, for any $\A$, $\y$, and positive (identical) initialization $\x(0) = \x_0 > \0$, that the Hadamard power $\xprod = \x^{\odot L}$ of the gradient flow $\x$ on \eqref{eq:LossPoly} converges to a solution $\x_+$ of \eqref{eq:NNLS}. \emph{A crucial point here is that, in contrast to theoretical results by \cite{chou2021more}, the existence of a solution $\A\x = \y$ is not required anymore.} In addition, we characterize the {decay rate of the loss} as $\mathcal O (1/\t)$. 

\begin{theorem} \label{theorem:L1_equivalence_positive}
    Let $\L\geq 2$, $\A\in\mathbb{R}^{\M\times\N}$ and $\y\in\mathbb{R}^{\M}$. Let $\x_0 > \0$ be fixed and let $\x(\t)$ follow the flow $\x'(t) = -\nabla \LossPoly(\x(\t))$ with $\x(0) = \x_0$.
    Let $S_+$ be the set defined in \eqref{eq:NNLS} and let $\xprod = \x^{\odot L}$.
    
    Then the limit $\xprodinfty:= \lim_{\t\to\infty} \xprod(\t)$ exists and lies in $S_+$. Also, there exists an absolute constant $C > 0$ that only depends on the choice of $\A$, $\y$, and $\x_0$ such that 
    \begin{align*}
        {\LossBasic(\xprod(\t)) - \LossBasic(\x_+) \le \frac{C}{\t},}
    \end{align*}
    for any $\t > 0$ {and any $\x_+ \in S_+$}. 
\end{theorem}

{
\begin{proof}[Proof sketch]
    By using properties of orthogonal projections, see Lemma \ref{lem:ConvexProjection}, existence of $\xprodinfty \in S_+$ can be obtained by slight modifications of the argument by \cite{chou2021more}.  Interestingly, we can streamline the settings $L=2$ and $L>2$ by relating the potential of \cite{chou2021more} to the \emph{Tsallis q-logarithm} \citep{tsallis1988possible}. For deducing the convergence rate, we introduce a Lyapunov argument that has not appeared in this context before to the best of our knowledge.
    The detailed proof of Theorem \ref{theorem:L1_equivalence_positive} is given in Appendix~\ref{sec:MainProofs}.
\end{proof}
\begin{remark}
    Note that non-negativity of the gradient flow trajectory follows from $\x_0 > 0$ and Picard-Lindelöf. Indeed, the Hadamard reparametrization induces $[\nabla \LossPoly(\x)]_i = 0$ for $x_i = 0$, see \eqref{eq:dynamicsBregman_positive}. If $x(t_*)_i = 0$, for $t_* > 0$, this would give rise to an alternative trajectory with $(x_0)_i = 0$ contradicting the local uniqueness of the flow around $t_*$. 
\end{remark}
}

%%%%%%%%%%%%%%%%%%%%%%%%%%%
\subsection{Convergence rate of accelerated reparametrized flow}
\label{sec:ConvergenceAcceleratedGF}
%%%%%%%%%%%%%%%%%%%%%%%%%%%

{In order to accelerate the reparametrized flow $\x'(t) = -\nabla \LossPoly(\x(\t))$ on $\eqref{eq:LossPoly}$, we note that} by setting $q = 2-\tfrac{2}{L} \in [1,2)$ and {recalling $\xprod = \x^{\odot L}$}, the reparametrized flow induces a Riemannian metric $\langle \u,\v \rangle_{\xprod} = \langle \xprod^{\odot (-q)}\odot \u,\v\rangle$, such that the corresponding Riemannian gradient of $\ell$ is given by
\begin{equation}\label{eq:gradRiemann}
    \nablaRiemann \LossBasic(\xprod)
    := \xprod^{\odot q} \odot \nabla \LossBasic(\xprod)
\end{equation}
{and we can replace the reparametrized flow on $\LossPoly$ by the Riemannian flow $\xprod'(\t) = - \L\nablaRiemann \LossBasic(\xprod)$ on $\LossBasic$.} 
A detailed derivation is provided in Appendix \ref{sec:Riemannian}.

\begin{remark}
\label{rem:Riemannian}
    {Interestingly, by using this Riemannian viewpoint, Theorem \ref{theorem:L1_equivalence_positive} can be derived from the results of \cite{li2022implicit}. At the same time, the proof we presented in Section \ref{sec:MainProofs} does not directly rely on tools and concepts from Riemannian geometry but leads to the same result.}
\end{remark}

Inspired by the Riemannian representation of the flow, we show significant improvement in the convergence rate of the flow by using a $2\textsuperscript{\textit{nd}}$-order ODE for \eqref{eq:LossPoly}. 
\begin{theorem}
\label{thm:Accelerated}
    Let $\L\geq 2$, $\A\in\mathbb{R}^{\M\times\N}$, $\y\in\mathbb{R}^{\M}$, and $\x_0 > \0$ be fixed.
    Moreover, let $\halfstep(\t) \in \mathbb{R}^N $ be a vector such that $(\halfstep(\t),\x(\t))$ satisfies the following dynamics
    \begin{equation}
    \label{accelerated-gd-continuous}
    \begin{split}
    \halfstep'(t) &= -\frac{t}{2} \halfstep(t)^{\odot q} \odot \nabla \LossBasic(\xprod(t)) \nonumber,\\
    \xprod'(t) &= \frac{2}{t}{\big(\halfstep(t) - \xprod(t) \big)},
    \end{split}
    \tag{ACC}
    \end{equation} 
with $\halfstep(0)=\xprod(0)=\xprod_0> \0$, where $q = 2-\frac{2}{L} {\in [1,2)}$ and $\LossBasic$ has been defined in \eqref{eq:LossBasic}.
Then, there exists an absolute constant $C > 0$ that only depends on {$q$ and the distance between $\xprod_0$ and $\x_+$} such that
    \begin{align*}
        {\LossBasic(\xprod(t))-\LossBasic(\x_+) \le \frac{C}{\t^2},}
    \end{align*}
    for any $\t > 0$ {and any $\x_+ \in S_+$ defined in \eqref{eq:NNLS}}. 
\end{theorem}

{
\begin{proof}[Proof sketch]
    The accelerated convergence rate is shown by a similar Lyapunov argument as in Theorem \ref{theorem:L1_equivalence_positive} using a higher-order Lyapunov function. The specific shape of \eqref{accelerated-gd-continuous} has been derived by reverse engineering the Lyapunov analysis of Theorem \ref{theorem:L1_equivalence_positive}.
    The full derivation of the result is given in Appendix \ref{sec:acc_proof}.
\end{proof}
}

\begin{remark}
     {Note that~\eqref{accelerated-gd-continuous} differs from a naive application of Nesterov's scheme to \eqref{eq:gradRiemann} since $\nabla \LossBasic$ is evaluated at $\xprod$ and not on $\halfstep$.}
\end{remark}

Our numerical simulations {in Section \ref{sec:accel}} show promising results when discretizing~\eqref{accelerated-gd-continuous}, even exceeding the predicted error decay of Theorem~\ref{thm:Accelerated}.

%%%%%%%%%%%%%%%%%%%%%%%%%%%
\subsection{Convergence rate of gradient descent}
\label{sec:ConvergenceGD}
%%%%%%%%%%%%%%%%%%%%%%%%%%%

We now proceed with a convergence rate analysis of reparametrized gradient descent on \eqref{eq:LossPoly}, {i.e., a discrete counterpart of Theorem \ref{theorem:L1_equivalence_positive}}. {Since our proof technique relies on a discretization of the Riemannian flow introduced in \eqref{eq:gradRiemann}, we directly work with GD defined on the products $\xprod = \x^{\odot L}$ and $\widetilde\nabla \LossBasic$ defined in \eqref{eq:gradRiemann}.}

\begin{theorem}
\label{thm:GD_ConvergenceRate_Improved}
Let ${\x_+} \in S_+$ where $S_+$ has been defined in \eqref{eq:NNLS}, and choose any $\xprod_0 > \0$ as initialization for the sequence $\xprod_k$ produced by
\begin{equation}\label{eq:discrete-gd}
    \xprod_{k+1}
    =\xprod_k - {\etagd_k\nablaRiemann \LossBasic(\xprod_k)}
\end{equation}
Fix $\delta > 0$, $q \in [1,2)$, $\gamma \in (0,1)$, { and 
\begin{align*}
    0 < \cgd \le \min \left\{ \frac{1}{\| \A^T \y \|_2^2 {(60 \mathfrak{C})}^{q}}, \frac{1}{\| \A^T\A \|_2 {(60 \mathfrak{C})}^{3q}} \right\}
\end{align*}
such that $c \in (0,1)$.} The constant {$\mathfrak{C}$ satisfies $1 \le \mathfrak{C} \le D_F({\x_+},\xprod_0) + c_{\max} + \| \x_+\|_\infty$ where $D_F$ denotes the Bregman divergence with respect to the \emph{Tsallis q-logarithm}, cf.\ Definition \ref{def:Bregman_Divergence} and \eqref{eq:Bregman2}, and} $c_{\max}$ is defined as

\begin{align*}
    c_{\max} :=
    \begin{cases}
        \zeta(2\gamma) & \text{for } \gamma > \frac{1}{2} \\
        2 \left( \log\big( \frac{2}{(2-q) \delta} \big) + (2-q)\delta \right) & \text{for } \gamma = \frac{1}{2} \\
        \frac{1}{1-2\gamma} \Big( \frac{2}{(2-q)\delta} \Big)^{\frac{1}{\gamma}-2} + 2 (2-q)\delta & \text{for } \gamma < \frac{1}{2}
    \end{cases},
\end{align*}
where $\zeta$ denotes the Riemann zeta function.

{Then, if} $\etagd_k := \frac{\cgd}{(k+2)^\gamma} $, we obtain the convergence rate
\begin{align*}
    \ell(\xprod_k) - \ell({\x_+})
    \le \frac{2 (2-q)^{-1} \cgd^{-1}}{k^{1-\gamma}} (  D_F({\x_+}, \xprod_0) + c_{\max} )
\end{align*}
for any $0 \le k < K$ where
\begin{align*}
    K := \min \{k \in \mathbb N \colon 
{2(\sqrt{\LossBasic(\xprod_k)} - \sqrt{\LossBasic(\x_+)})} < \sqrt{\delta} \}.
\end{align*}
In addition, we have $\xprod_k \in \R_{>0}^\N$ with $\| \xprod_k \|_2 \le \mathfrak{C}$. 
\end{theorem}

{
\begin{proof}[Proof sketch]
    To deduce the convergence rate, we carefully discretize the Lyapunov argument from Theorem~\ref{theorem:L1_equivalence_positive} using Taylor's theorem. This results in an intricate inductive argument, which simultaneously controls three quantities: (i) the $\ell_2$-norms of the iterates, (ii) the entry-wise signs of the iterates, and (iii) the decay in the Lyapunov function.
    The detailed proof of Theorem \ref{thm:GD_ConvergenceRate_Improved} is in Appendix \ref{sec:gd_proof}.
\end{proof}
}

\begin{remark} 
    Note that for quickly decaying stepsizes, i.e., $\gamma > \frac{1}{2}$, we may set $\delta=0$ in Theorem \ref{thm:GD_ConvergenceRate_Improved} which yields {$K=\infty$ and thus} convergence of the sequence of iterates to global optimality. For larger stepsizes, the result only controls the decay rate of the iterates until $\LossBasic(\xprod_k) - \LossBasic({\x_+}) = \mathcal O ({\sqrt{\delta}})$, i.e., a neighborhood of $S_+$ is reached. {Moreover, Figure \ref{fig:GammaComparison} illustrates that the predicted rate of $\mathcal{O}(1/k^\gamma)$ also holds for constant stepsize, i.e., $\gamma = 0$. Our restriction to $\gamma \in (0,1)$ and the neighborhood restrictions are probably artifacts of our proof technique.}
\end{remark}

{We leave it as an open question for future work to analyze a discrete analog of the accelerated flow in~\eqref{accelerated-gd-continuous}}.

%%%%%%%%%%%%%%%%%%%%%%%%%%%
\subsection{NNLS for sparse recovery}
\label{sec:NNLSforSparseRecovery}
%%%%%%%%%%%%%%%%%%%%%%%%%%%

It is known in the field of sparse recovery that under certain assumptions on the measurement operator $\A$, solving \eqref{eq:NNLS} promotes sparsity of its solution without any hyper-parameter tuning and comes with improved robustness. This observation dates back to the seminal works of \cite{donoho2005sparse,donoho2010counting}. Therefore, it is expected that the same would happen for GD applied to the overparametrized function in \eqref{eq:LossPoly}. As a nice by-product of reusing existing theory in our proofs, we observe that if $\x_0$ is chosen sufficiently close to zero, the limit of gradient flow (approximately) minimizes the $\ell_1$-norm among all possible solutions of \eqref{eq:NNLS}, which has previously been observed by \cite{chou2021more}. {Theorem \ref{theorem:L1_equivalence_positive_merged} in Appendix \ref{sec:MainProofs} provides the detailed relation between initialization scale and implicit $\ell_1$-regularization. The precise statements on how to use this for sparse recovery} are given in Appendix \ref{sec:NNLS_sparse}. We will only treat the result informally here. According to Theorem \ref{theorem:L1_equivalence_positive_PartII},
\begin{align*}
    \|\xprodinfty\|_1 - \min_{\z \in S_+} \|\z\|_1 \leq \epsilon,
\end{align*}
where $\epsilon$ decreases as the initialization scale decreases. If we further assume that $\A \in \R^{\N \times \M}$ has the $\ell_2$-robust null space property (Definition \ref{def:NSP:statement}) and $\ell_1$-quotient property (Definition \ref{def:l1_quotient}), then we can stably recover sparse ground truths (Theorem \ref{theorem:robustness}).

Theorem \ref{theorem:robustness} implies that our approach to solving \eqref{eq:NNLS} is stable with respect to negative entries of the ground truth, i.e., the reconstruction error depends on the sparsity of the positive part $\x_+$ of $\x_*$ and the magnitude of the negative part $\x_-$ of $\x_*$. The experiments we perform in Section \ref{sec:NumericsStability} suggest that the established solvers for \eqref{eq:NNLS} are less stable under such perturbations.

{We finally mention that if one is using NNLS for sparse recovery, then an alternative approach that leverages Hadamard overparametrization is the one of \cite{poon2023smooth} in which the authors smoothen an $\ell_1$-regularized objective by Hadamard overparametrization combined with explicit $\ell_2$-regularization. Due to the explicit use of regularization terms, however, this is different from the implicit bias setting analyzed in this work and clearly designed for sparse recovery.}

%%%%%%%%%%%%%%%%%%%%%%%%%%%
\section{NUMERICAL EXPERIMENTS}
\label{sec:Numerics}
%%%%%%%%%%%%%%%%%%%%%%%%%%%

Let us finally turn to a numerical evaluation of our theoretical insights.\footnote{Our code is available at \url{https://github.com/claudioverdun/NNLS_AISTATS_2025}.}
We compare the following six methods for solving NNLS here:

\begin{itemize}
    \item \textbf{GD-$n$L:} Vanilla gradient descent applied to $\mathcal L$ in \eqref{eq:LossPoly} with $n$ layers, for $n \in \mathbb N$. This is the discretized version of the gradient flow we considered in Section \ref{sec:MainResults}. As initialization we use $\alpha \bo$, for $\alpha > 0$.\footnote{Since all factors $\x^{(k)}$ stay identical over time, cf.\ Remark \ref{remark2.2}, one only needs to store $\x^{(1)}$ and there is no memory overhead of \textbf{GD-$n$L} in comparison to GD on \eqref{eq:NNLS}.}
    \item \textbf{SGD-$n$L:} Stochastic gradient descent applied to $\mathcal L$ in \eqref{eq:LossPoly} with $n$ layers, for $n \in \mathbb N$. A probabilistic variant of \textbf{GD-$n$L}. In the experiments we use $\M/10$ as batch-size for \textbf{SGD-$n$L} and initialize with $\alpha \bo$, for $\alpha > 0$.
    \item \textbf{CVX-NNLS:} Solving the quadratic formulation of  \eqref{eq:NNLS} described in \eqref{eq:QuadraticFormulation} via ADMM. We use the Python-embedded modeling language CVX \footnote{\url{https://www.cvxpy.org/}} which, in turn, uses the solver OSQP \footnote{\url{https://osqp.org/docs/solver/index.html}} for NNLS.
    \item \textbf{LH-NNLS:} The standard Python NNLS-solver \emph{scipy.optimize.nnls}, which is an active set method and is based on the original Lawson-Hanson method \citep{lawson1995solving}. It is not scalable to high dimensions since it requires solving linear systems in each iteration. (An accelerated version of \textbf{LH-NNLS} is provided by \cite{bro1997fast}. Since both methods produce the same outcome in our experiments, we only provide the results for \textbf{LH-NNLS}.)
    \item \textbf{TNT-NN:} An alternative but more recent active set method that heuristically works well and dramatically improves over \textbf{LH-NNLS} in performance \citep{myre2017tnt}. We used the recent Python implementation available at \url{https://github.com/gdcs92/pytntnn}.
    \item \textbf{PGD}: Projected gradient descent for NNLS as described by \cite{polyak2015projected}.
\end{itemize}

We {begin by examining} the convergence rates of \textbf{GD-$n$L} and \textbf{SGD-$n$L} for various choices of step size, including constant step-size and Nesterov acceleration.
We also compare it to the Barzilai-Borwein (BB) stepsize rule due to its popularity and its numerical efficiency.
Next, we analyze gradient acceleration corresponding to the differential equation~\eqref{accelerated-gd-continuous}, with $L=2$, and illustrate its performance in comparison to the discretized accelerated gradient method for the dynamics in the original variable $\x'(t) = - \nabla \LossPoly(\x)$.
We proceed to the analysis of stepsize decay rates, as required by Theorem~\ref{thm:GD_ConvergenceRate_Improved}.
We then investigate the role of initialization and the number of layers on the implicit sparsity retrieval behavior of \textbf{GD-$n$L}.
Finally, we compare the stability of the proposed overparametrized methods against those aforementioned established methods when recovering perturbed signals that are not strictly non-negative. Additional large-scale experiments are provided in Appendix \ref{sec:AdditionalNumerics}.

%%%%%%%%%%%%%%%%%%%%%%%%%%%
\subsection{Different stepsizes}
\label{sec:NumericsStepsize}
%%%%%%%%%%%%%%%%%%%%%%%%%%%

In the first experiment, we compare the convergence rates of \textbf{GD-$n$L} and \textbf{SGD-$n$L} for various choices of step size.
Apart from using a constant step-size $\eta = 10^{-2}$, we also consider Nesterov acceleration \citep{nesterov1983method} and the BB stepsize rules \citep{fletcher2005barzilai,raydan1997barzilai},
and start all methods from $\x = \alpha \cdot \bo$, with $\alpha = 10^{-2}$.
Figure~\ref{fig:compare} shows the decay in training error, i.e., objective value over time in two different settings: the dense case, i.e., we have a squared system with $\M = \N = 50$ and a dense ground-truth, and the sparse case, i.e., we have an underdetermined system with $\M = 30$, $\N = 50$, and a $3$-sparse ground-truth.
In both settings, $\A$ has Gaussian entries.
As Figures \ref{fig:compareDense} and \ref{fig:compareSparse} show, advanced choices of step size notably improve the convergence rate of the gradient methods.
Moreover, \textbf{GD-$n$L} seems to profit more from the acceleration than \textbf{SGD-$n$L}. 

\begin{figure}[ht]
    \centering
    \begin{subfigure}[b]{0.45\textwidth}
        \centering
        \includegraphics[width=.9\textwidth]{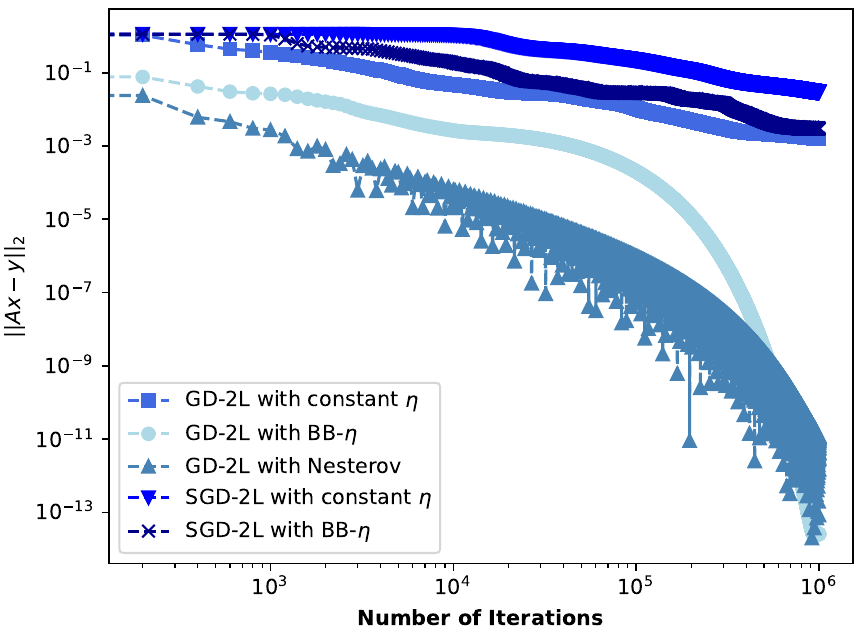}
        \subcaption{Dense case}
        \label{fig:compareDense}
    \end{subfigure}
    \quad
    \begin{subfigure}[b]{0.45\textwidth}
        \centering
        \includegraphics[width=.9\textwidth]{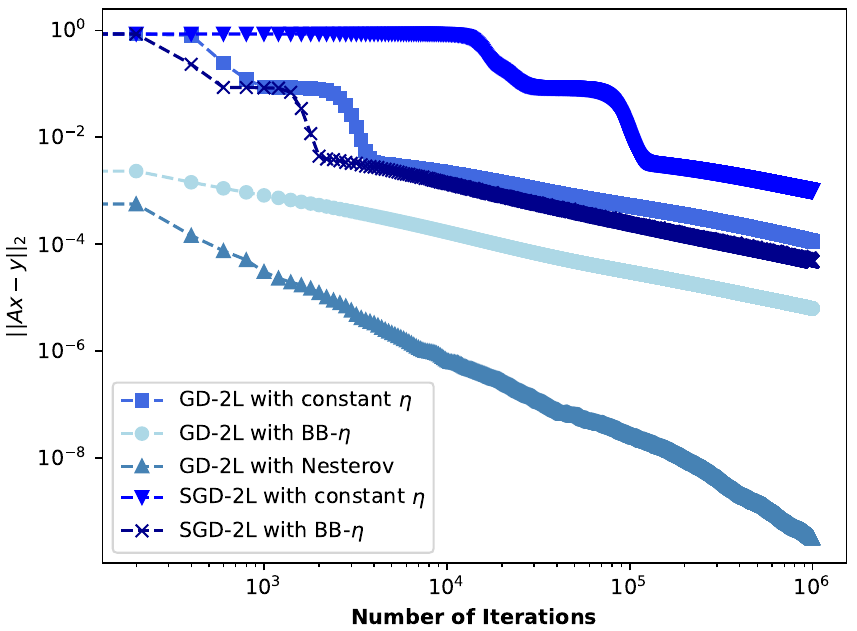}
        \subcaption{Sparse case}
        \label{fig:compareSparse}
    \end{subfigure}
    \caption{Convergence rate for various choices of step-size, see Section \ref{sec:NumericsStepsize}.}
    \label{fig:compare}
\end{figure}

%%%%%%%%%%%%%%%%%%%%%%%%%%%
\subsection{Acceleration}
\label{sec:accel}
%%%%%%%%%%%%%%%%%%%%%%%%%%%

In our next experiment, we analyze gradient acceleration corresponding to the differential equation~\eqref{accelerated-gd-continuous}, with $L=2$, and a $20\times50$ linear system recovering a $3$-sparse ground truth.
We first solve this equation with an adaptive Runge-Kutta method~\citep{dormand1980rk}, to evaluate its convergence rate, which turn out to be faster than predicted, following an $\mathcal{O}(\t^{-3})$ regime, as observed in Figure~\ref{fig:AcceleratedODE}.
The ODE was started at $\xprod_0 = 10^{-3} \cdot \1$.

Then, we compare two accelerated versions of gradient descent, the first coming from the Riemannian interpretation of equation~\eqref{accelerated-gd-continuous}, in the variables $\xprod$; and a second from Nesterov's scheme directly in the variables $\x$.
We start both methods from the same point as the ODE, and use $\eta = 0.1$;
detailed step rules are provided in Appendix~\ref{sec:AccelerationDiscretization}.
In Figure~\ref{fig:AcceleratedGD}, we observe that the accelerated Riemannian method reproduces the same convergence rate as the ODE.
Moreover, and remarkably, Nesterov's accelerated gradient descent in the coordinates $\x$ simultaneously features faster convergence and smaller oscillations compared to both the ODE solution and the first accelerated scheme.

\begin{figure}[ht]
    \centering
    \begin{subfigure}[b]{0.475\textwidth}
        \centering
        \includegraphics[width=0.95\textwidth]{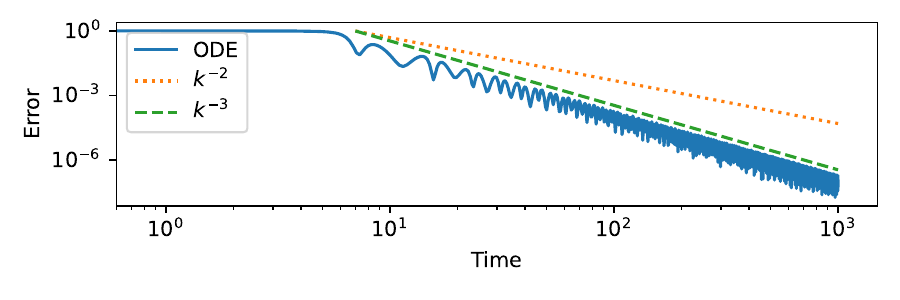}
        \vskip -2ex
	    \subcaption{ODE solution}
	    \label{fig:AcceleratedODE}
    \end{subfigure}
    \vskip 2ex
    
    \begin{subfigure}[b]{0.475\textwidth}
        \centering
        \includegraphics[width=0.95\textwidth]{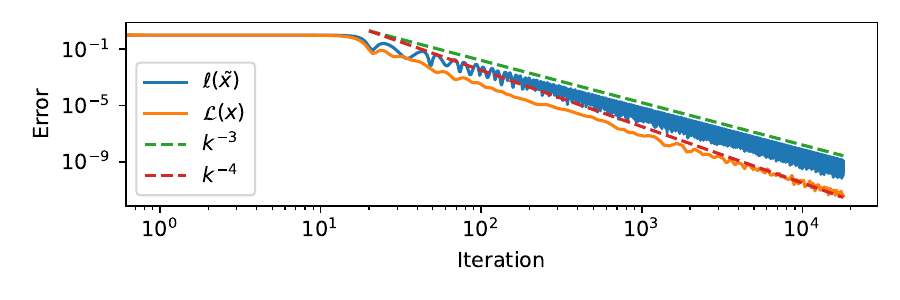}
        \vskip -2ex
	    \subcaption{Acceleration methods.}
	    \label{fig:AcceleratedGD}
    \end{subfigure}
    \caption{Accelerated gradient methods: (a) Numerical solution of ODE~\eqref{accelerated-gd-continuous} and (b) behavior of discretized accelerated gradient.}
    \label{fig:Acceleration}
\end{figure}

%%%%%%%%%%%%%%%%%%%%%%%%%%%
\subsection{Stepsize decay}
\label{sec:NumericsStepDecay}
%%%%%%%%%%%%%%%%%%%%%%%%%%%

Now, we illustrate the choice of decay rates $\eta = C k^{-\gamma}$ from Theorem~\ref{thm:GD_ConvergenceRate_Improved} for \textbf{GD-$n$L}.
We set $\alpha = 0.1$ and $\eta = 0.1$ for these experiments.
The dashed lines in Figure~\ref{fig:GammaComparison} correspond to predicted decay rates $\mathcal{O}(k^{1 - \gamma})$, which are in excellent agreement in our experiment, where $M=20$, $N=50$ and $x \in \R^N$ is a 3-sparse vector.
It is remarkable that the error decreases in almost discrete steps, which are ``just-in-time'' to keep the global decay rates and that the effect is stronger for larger $L$. {This step-wise decay is related to saddle-to-saddle dynamics as observed by \cite{berthier2023incremental,pesme2023saddle,chou2020implicit}.}

\begin{figure}[ht]
    \centering
    \begin{subfigure}[b]{0.475\textwidth}
        \centering
        \includegraphics[width=.95\textwidth]{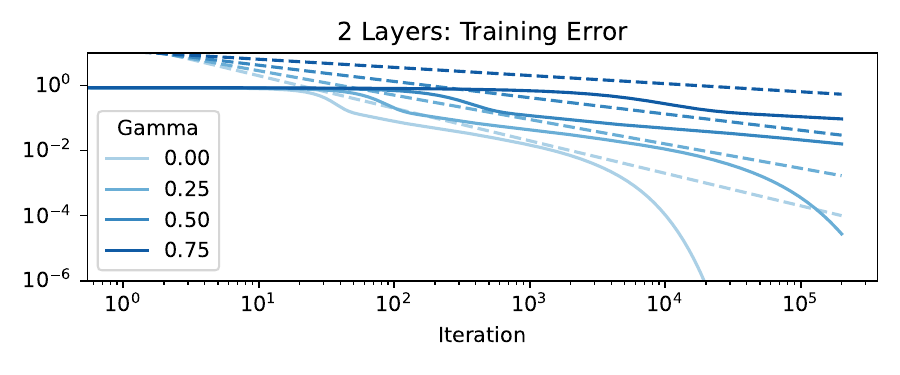}
        \vskip -2ex
        \subcaption{2 Layers}
        \label{fig:compare_gamma_2L}
    \end{subfigure}
    \vskip 2ex
    
    \begin{subfigure}[b]{0.475\textwidth}
        \centering
        \includegraphics[width=.95\textwidth]{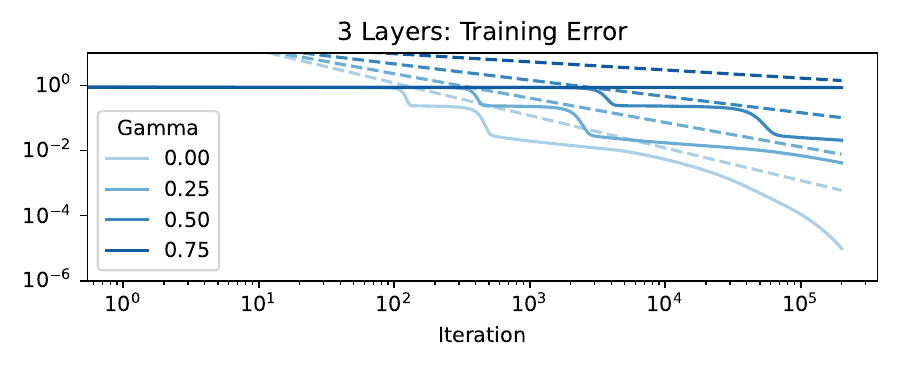}
        \vskip -2ex
        \subcaption{3 Layers}
        \label{fig:compare_gamma_3L}
    \end{subfigure}
    \caption{Impact of decaying stepsize rate $\gamma$, for (a) $L=2$ and (b) $L=3$.}
    \label{fig:GammaComparison}
\end{figure}

%%%%%%%%%%%%%%%%%%%%%%%%%%%
\subsection{Initialization and Number of Layers}
\label{sec:NumericsInit}
%%%%%%%%%%%%%%%%%%%%%%%%%%%

In the fourth experiment, we validate the $\ell_1$-norm regularization that, according to Theorem \ref{theorem:L1_equivalence_positive_PartII}, can be induced by using a small initialization for \textbf{GD-$n$L}.
For $\M = 10$ and $\N = 50$, we draw a random Gaussian matrix $\A\in \R^{\M\times \N}$, create a $3$-sparse ground-truth $\x \in \R^\N$, and set $\y = \A\x$.
Figure~\ref{fig:AlphaComparison} depicts the $\ell_1$-norm of the limits of \textbf{GD-$2$L} and \textbf{GD-$3$L}, for $10^{-3} \le \alpha \le 10^{-1}$ and constant step-size $\eta = 10^{-1}$.
As a benchmark, the $\ell_1$ minimizer among all feasible solutions is computed via basis pursuit (\textbf{L1-MIN}). Figure \ref{fig:AlphaComparison} shows that \textbf{GD-$2$L} and \textbf{GD-$3$L} converge to $\ell_1$-norm minimizers if $\alpha$ is sufficiently small.
As predicted by Theorem \ref{theorem:L1_equivalence_positive_PartII}, the requirements on $\alpha$ to allow such regularization are milder for the $3$-layer case \textbf{GD-$3$L}.
Let us finally mention that neither \textbf{LH-NNLS} nor \textbf{CVX-NNLS} reach $\ell_1$-minimality; the former comes close, but the latter arrives at a solution with $\ell_1$-norm of $3.514$.
Seemingly, the matrix $\A$, although sufficiently well-behaved for sparse recovery in general, does not guarantee the uniqueness of the NNLS solution here, {cf.\ Section \ref{sec:NNLS_sparse}}.

\begin{figure}[ht]
    \centering
    \begin{subfigure}[b]{0.45\textwidth}
        \centering
        \includegraphics[width=0.95\textwidth]{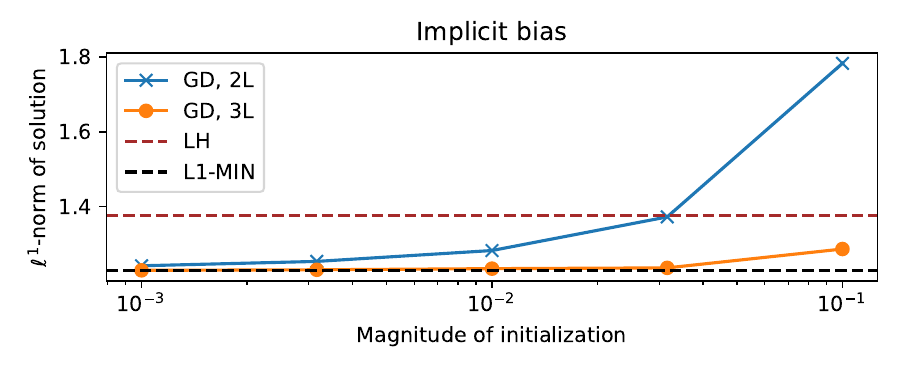}
        \vskip -2ex
	    \subcaption{Initialization.}
	    \label{fig:AlphaComparison}
    \end{subfigure}
    \vskip 2ex
    
    \begin{subfigure}[b]{0.45\textwidth}
        \centering
        \includegraphics[width=0.95\textwidth]{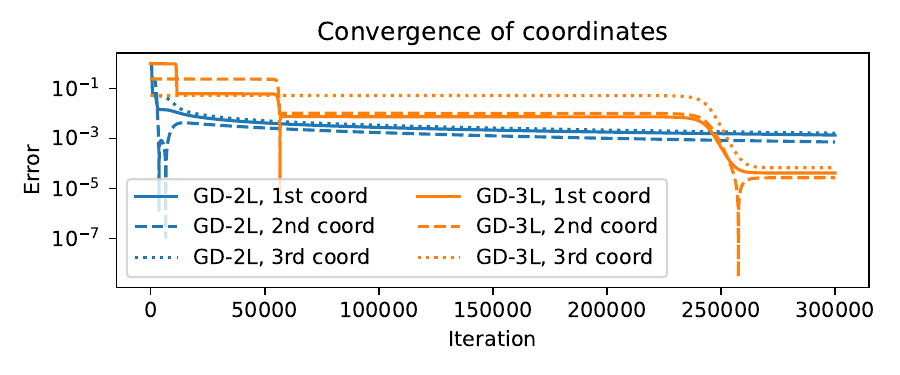}
        \vskip -2ex
	    \subcaption{Number of layers.}
	    \label{fig:LayerComparison}
    \end{subfigure}
    \caption{Influence of initialization and number of layers on \textbf{GD-$n$L}, cf.\ Section \ref{sec:NumericsInit}.}
    \label{fig:InitAndStepsize}
\end{figure}

In a further experiment, we take a closer look at the reconstruction behavior of \textbf{GD-$n$L}.
In particular, we compare how different entries of the ground truth are approximated over time for $n=2$ and $n=3$ layers.
For initialization magnitude and step-size we choose $\alpha = 10^{-2}$ and $\eta = 10^{-2}$.
We set $\M = 30$ and $\N = 50$, draw a random Gaussian matrix $\A\in \R^{\M\times \N}$, create a $3$-sparse non-negative ground-truth $\x \in \R^\N$, and set $\y = \A\x$. Note that $\A$ satisfies the assumptions of Theorem \ref{theorem:robustness}.
Figure~\ref{fig:LayerComparison} depicts the entry-wise error between the three non-zero ground-truth entries and the corresponding entries of the iterates of \textbf{GD-$2$L} and \textbf{GD-$3$L}.
As already observed in previous related works, we see that a deeper factorization leads to sharper error transitions that occur later and that more dominant entries are recovered faster than the rest.
Interestingly, some kind of overshooting occurs in the dominant entries: dashed curves, for the second-highest coordinate, suggest that \textbf{GD-$n$L} does not monotonically decrease the error in all components but rather concentrates heavily on the leading component(s) in the beginning and only starts distributing the error over time.
In this way, \textbf{GD-$n$L} could be interpreted as a self-correcting greedy method.

%%%%%%%%%%%%%%%%%%%%%%%%%%%
\subsection{Stability with Negative Entries}
\label{sec:NumericsStability}
%%%%%%%%%%%%%%%%%%%%%%%%%%%

\begin{figure*}
\captionsetup[subfigure]{justification=centering}
    \centering
    \begin{subfigure}[b]{0.19\textwidth}
        \centering
        \includegraphics[width = .5\textwidth]{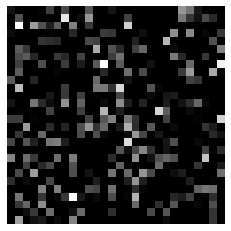}
        \subcaption*{HL-NNLS \\ ($q=0.5$)}
    \end{subfigure}
    \hfill
    \begin{subfigure}[b]{0.19\textwidth}
        \centering
        \includegraphics[width = .5\textwidth]{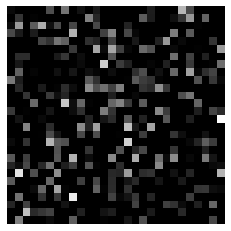}
        \subcaption*{TNT-NN \\ ($q=0.5$)}
    \end{subfigure}
    \hfill
    \begin{subfigure}[b]{0.19\textwidth}
        \centering
        \includegraphics[width = .5\textwidth]{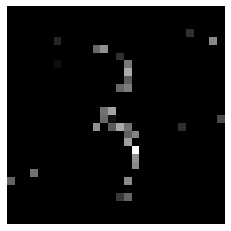}
        \subcaption*{SGD-3L \\ ($q=0.5$)}
    \end{subfigure}
    \hfill
    \begin{subfigure}[b]{0.19\textwidth}
        \centering
        \includegraphics[width = .5\textwidth]{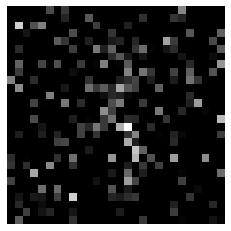}
        \subcaption*{GD-3L \\ ($q=0.5$)}
    \end{subfigure}
    \hfill
    \begin{subfigure}[b]{0.19\textwidth}
        \centering
        \includegraphics[width = .5\textwidth]{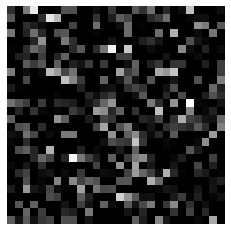}
        \subcaption*{PGD \\ ($q=0.5$)}
    \end{subfigure}\\
    \vskip 3mm
    \begin{subfigure}[b]{0.19\textwidth}
        \centering
        \includegraphics[width = .5\textwidth]{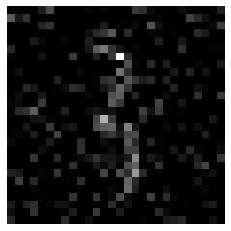}
        \subcaption*{HL-NNLS \\ ($q=0.05$)}
    \end{subfigure}
    \hfill
    \begin{subfigure}[b]{0.19\textwidth}
        \centering
        \includegraphics[width = .5\textwidth]{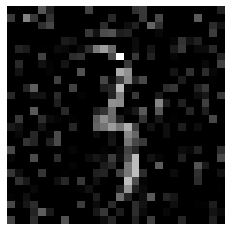}
        \subcaption*{TNT-NN \\ ($q=0.05$)}
    \end{subfigure}
    \hfill
    \begin{subfigure}[b]{0.19\textwidth}
        \centering
        \includegraphics[width = .5\textwidth]{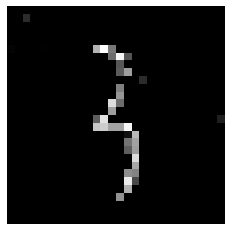}
        \subcaption*{SGD-3L \\ ($q=0.05$)}
    \end{subfigure}
    \hfill
    \begin{subfigure}[b]{0.19\textwidth}
        \centering
        \includegraphics[width = .5\textwidth]{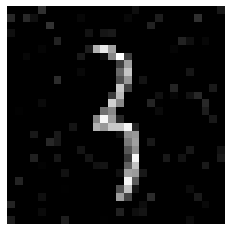}
        \subcaption*{GD-3L \\ ($q=0.05$)}
    \end{subfigure}
    \hfill
    \begin{subfigure}[b]{0.19\textwidth}
        \centering
        \includegraphics[width = .5\textwidth]{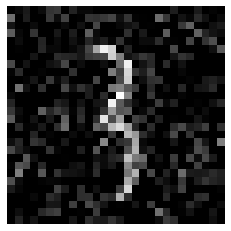}
        \subcaption*{PGD \\ ($q=0.05$)}
    \end{subfigure}
    \caption{Illustration of the MNIST reconstruction, see Section \ref{sec:NumericsStability}.
    }
    \label{fig:MNIST_imag}
\end{figure*}

In our final experiment, we compare the stability of the proposed methods when recovering perturbed signals that are not strictly non-negative anymore. We consider two different set-ups. We use images from the MNIST data set and CIFAR10. In Appendix \ref{sec:NumericsStability_gaussian}, we analyze the perturbed behavior with generic random signals. In both scenarios we set $\alpha=10^{-2}$ and $\eta = 10^{-2}$ for \textbf{GD-$n$L} and \textbf{SGD-$n$L}. It is noteworthy that in all tested instances \textbf{GD-$n$L} and \textbf{SGD-$n$L} outperform the established methods in reconstruction quality, cf.\ Figure \ref{fig:MNIST_imag}. Whereas \textbf{HL-NNLS} and \textbf{TNT-NN} are designed to retrieve only non-negative signals, the gradient-based methods can stably deal with vectors that have small negative components by not using explicit constraints.

\textbf{MNIST signals.} Again $\A\in\mathbb{R}^{M\times N}$ is a random Gaussian matrix, where now $\M=300$ and $\N=28^2=784$. We take $\x_+$ to be the original MNIST image (number three) and define a corrupted signal $\x$ including negative (mean zero) Gaussian noise $\x_-$. For $q \in [0,1]$, we scale $\x_+,\x_-$ such that
\begin{equation}\label{eq:q_level}
    \|\x_+\|_2^2 = 1-q
    \quad \text{and} \quad
    \|\x_-\|_2^2 = q.
\end{equation}
The perturbed signal is then given by $\x = \x_+ - \x_-$, and $q$ regulates the negative corruption. We regard the copy of $\x_+$ scaled to $\ell_2$-norm $(1-q)$ as ground truth and, by abuse of notation, also refer to it as $\x_+$. The corresponding measurements are given as $\y = \A\x = \A(\x_+ - \x_-)$.\\
These figures clearly show that the gradient descent-based methods outperform the established NNLS solvers. Only for small negative noise levels and MNIST data \textbf{SGD-$3$L} yields worse results than \textbf{HL-NNLS} and \textbf{TNT-NN}. As can be seen from Figure \ref{fig:MNIST_imag}, this worse error is mainly caused by incorrect values on the support of $\x_+$. Visually, even for small $q$, the MNIST reconstruction of \textbf{SGD-$3$L} is far better than the one of \textbf{HL-NNLS} and \textbf{TNT-NN}. The experiment also reveals two interesting points. First, whereas it has numerically been observed by \cite{pesme2021implicit} that, compared to gradient descent, stochastic gradient descent reduces the generalization (resp.\ approximation) error if measured in $\ell_2$-norm, we observe this (in the case of NNLS) only for large values of $q$, i.e., strong negative perturbations. For small values of $q$, Figure \ref{fig:compare} rather suggests that \textbf{GD-$3$L} outperforms \textbf{SGD-$3$L}. Second, for all values of $q$, the solutions computed by \textbf{SGD-$3$L} are visually closer to the ground truth than the ones computed by \textbf{GD-$3$L}, cf.\ Figure \ref{fig:MNIST_imag}. This suggests that, even in the simple context of sparse recovery, (i) the $\ell_2$-norm might not be the appropriate measure for the generalization error and (ii) the stochasticity in \textbf{SGD-$3$L} apparently improves the generalization quality. Formalizing and proving this observation is an appealing topic for future research. We also show that, despite fast, \textbf{PGD} does not induce an implicit bias as prominent as \textbf{GD-$3$L} or \textbf{SGD-$3$L}.

\textbf{CIFAR10 signals.}
We conduct similar experiments on perturbed CIFAR10 images in gray-scale, as shown in Figure \ref{fig:cifar10}. Those images are in general less sparse than MNIST images, which affects the performance of each algorithm. The most significant difference between reconstructions from MNIST and from CIFAR10 is the outcome of SGD. For CIFAR10, GD yields the best reconstruction while SGD {induces a too strong sparsity regularization and} yields the worst reconstruction.

%%%%%%%%%%%%%%%%%%%%%%%%%%%
\section{CONCLUSION}
\label{sec:Conclusion}
%%%%%%%%%%%%%%%%%%%%%%%%%%%
\vspace{-.2cm}
In this paper, we have shown that vanilla gradient flow/descent on the least squares objective can trade constraints for the complexity of the objective function. We use that to scalably solve NNLS. Our approach comes with strong theoretical guarantees, including global convergence to NNLS solutions at an $\mathcal{O}(1/t)$ rate for gradient flow and an $\mathcal{O}(1/k)$ rate for gradient descent with decaying step sizes. Moreover, we prove that an accelerated gradient flow achieves an improved $\mathcal{O}(1/t^2)$ rate and we numerically show that this scheme attains an $\mathcal{O}(1/k^{3})$ decay. These results provide a new perspective on solving NNLS by leveraging the implicit bias of gradient descent, avoiding limitations such as the need for careful step size tuning and the lack of theoretical guarantees that affect some existing solvers. Our analysis also reveals intriguing connections between the overparametrization literature and classical optimization problems. We believe the techniques developed here, trading geometry by reparametrization, could potentially be extended to other constrained optimization problems, opening up an exciting new paradigm for optimization algorithm design.\\
There remain many promising directions for future work, including extending {our discrete analysis to the accelerated case} and exploring applications to large-scale problems in areas such as non-negative matrix factorization. More broadly, we see great potential in exploiting the implicit bias of gradient descent, a phenomenon that has primarily been studied for explaining neural network training, to design new optimization methods for classical numerical problems.

\vspace{1.5cm}

\bibliographystyle{unsrtnat}
\bibliography{ref}

\clearpage

\appendix

\onecolumn
\aistatstitle{Supplementary Materials to the paper ``Get rid of your constraints and reparametrize:
A study in NNLS and implicit bias''}

This supplementary material provides additional details, proofs, and experiments to support the main paper. The contents are organized as follows:

In Appendix \ref{sec:RelatedWork}, we provide an extensive review of related work, covering both the literature on overparameterization and implicit regularization as well as previous approaches to solving non-negative least squares problems.
Appendix \ref{sec:ProofRemark1} offers a detailed proof of Remark \ref{remark2.2}, establishing the connection between our approach and overparameterized models.
Appendix \ref{sec:MainProofs} contains the complete proof for Theorem \ref{theorem:L1_equivalence_positive}, which is central to our convergence analysis.
In Appendix \ref{sec:Riemannian}, we derive the Riemannian metric and Riemannian gradient associated with our reparameterized flow.
Appendix \ref{sec:acc_proof} presents the proof of Theorem \ref{thm:Accelerated}, demonstrating the accelerated convergence rate for our proposed continuous dynamics.
Appendix \ref{sec:gd_proof} provides a comprehensive proof of Theorem \ref{thm:GD_ConvergenceRate_Improved}, establishing convergence rates for discrete gradient descent.
In Appendix \ref{sec:NNLS_sparse}, we explore the application of our method to sparse recovery problems, including the proof of Theorem \ref{theorem:robustness} on stability as well as the proof of Remark \ref{rem:WeightedL1}, discussing weighted $\ell_1$-norm regularization.
Finally, Appendices \ref{sec:AccelerationDiscretization} and \ref{sec:AdditionalNumerics} present implementation details of the accelerated methods and additional numerical experiments, including large-scale problems and comparisons with projected gradient descent, to further validate our theoretical results and demonstrate the practical efficacy of our approach.

\tableofcontents

%%%%%%%%%%%%%%%%%%%%%%%%%%%
\section{RELATED WORK --- EXTENDED}
\label{sec:RelatedWork}
%%%%%%%%%%%%%%%%%%%%%%%%%%%

The results we presented in Section \ref{sec:MainResults} unite two very different and (up to this point) independent lines of mathematical research: the decade-old question of how to solve NNLS in an efficient way and the rather recent question of what kind of implicit bias gradient descent exhibits.

%%%%%%%%%%%%%%%%%%%%%%%%%%%
\subsection{Related work --- Overparameterization and implicit regularization}
\label{sec:ImplicitBias}
%%%%%%%%%%%%%%%%%%%%%%%%%%%

The recent success of overparametrization in machine learning models \citep{goodfellow2014GAN} has raised the question of how a plain algorithm like (stochastic) gradient descent can succeed in solving highly non-convex optimization problems like the training of deep neural networks and, in particular, find ``good'' solutions, i.e., parameter configurations for which the network generalizes well to unseen data. The objective functions in such tasks typically have infinitely many global minimizers (usually there are infinitely many networks fitting the training samples exactly in the overparameterized regime \citep{zhang2017rethink}), which means that the choice of algorithm strongly influences which minimum is picked. This intrinsic tendency of an optimization method towards minimizers of a certain shape has been dubbed ``implicit bias'' and, in the case of gradient descent, has led to an active line of research in the past few years.\\
Based on numerical simulations, several works \citep{neyshabur2017geometry,neyshabur2015,zhang2017rethink,zhang2021understanding} systematically studied this implicit bias of gradient descent in the training of deep neural networks. Whereas it is not even clear by now how to measure the implicit bias --- one could quantify it, e.g., in low complexity \citep{neyshabur2015} or in high generalizability \citep{hochreiter1997flat} ---, the empirical studies observed that the factorized structure of such networks is crucial for successful training. However, due to the complexity of the model class, there is still no corresponding thorough theoretical analysis/understanding available. To close the gap between theory and practice, simplified ``training'' models have been proposed and analyzed in a great number of works: vector factorization for data with \citep{vaskevicius2019implicit,zhao2019implicit} or without \citep{arora2018optimization,gunasekar2019implicit,chou2021more} restricted isometry property, and matrix factorization with deep structure \citep{arora2019implicit,chou2020implicit,gunasekar2017implicit,neyshabur2017geometry,neyshabur2015} or shallow structure  \citep{geyer2019implicit,Bach2019implicit,Gissin2019Implicit,stoger2021small}, kernel-like models \citep{soudry2018implicit,woodworth2020kernel}, mirror descent\citep{wu2021implicit}, together with techniques such as balancing \citep{wang2022large}, early stopping \citep{li2021implicit}, and analysis regarding to the robustness of those models \citep{liu2022robust,you2020robust}.

Essentially, all of these works agree in the point that, if initialized close to zero and applied to a plain least-squares formulation of factorized shape, cf.\ $\Lover$ in \eqref{eq:Lover}, vanilla gradient descent/flow exhibits an implicit bias towards global minimizers that are sparse (vector case) resp.\ of low-rank (matrix case). This is remarkable since it shows that in overparametrized regimes, gradient descent has an implicit tendency toward ``simple'' solutions.
In particular, it has been observed that such an implicit bias of gradient descent can be explained by three complementary perspectives: implicit reparametrization of the underlying gradient flow \citep{neyshabur2017geometry,li2022implicit,chou2023induce}, mirror flow \citep{krichene2015accelerated,gunasekar2021mirrorless,li2022implicit}, and Riemannian gradient flow \citep{li2022implicit}. One can roughly distinguish three phases on the way from the first partial results that aimed to understand the surprising properties of GD in the overparametrized setting to a general research program focused on the theory of implicit bias: (i.) In the beginning, researchers characterized the limit \citep{gunasekar2017implicit,arora2019implicit} assuming convergence; in many instances, they needed additional assumptions such as infinitesimal initialization. Such a characterization can also be analyzed from a mirror flow perspective  \citep{krichene2015accelerated,gunasekar2021mirrorless,li2022implicit}. (ii.) Once the characterization has been established, many works started to derive sufficient conditions for convergence \citep{Bach2019implicit,chou2021more,li2022implicit} and extended the results to a range of (small) initialization as opposed to infinitesimal initialization. (iii.) Most of the above-mentioned works only analyze gradient flow since it is in general difficult to establish convergence rate results for GD in the overparametrized regime due to non-convexity in the loss function. In principle, recent work on reparametrized gradient flow \citep{chou2023induce} suggests linear convergence of gradient methods on overparametrized linear regression. Due to a strong dependence of the therein appearing constants on the conditioning of the optimization problem, however, the obtained linear rates do not explain the convergence rates observed in practice. Another direction is to take step sizes as large as possible to enter the edge of stability regime  \citep{arora2022understanding,wu2023implicit,wang2023good} for which hardly any theory exists so far. 
In the infinite width regime, such results can even be connected to general neural networks via the neural tangent kernel \citep{jacot2018NTK}. However, fully understanding the implicit bias of GD in the training of finite-width networks remains a challenge.

Whereas the obtained insights on vector and matrix factorization might not yet resolve the mystery of deep learning, they provide valuable tools for more classical problems. 
One such example is sparse recovery, which lies at the interface of high-dimensional statistics and signal processing. Over the past two decades, several methods have been developed to recover intrinsically low-dimensional data, such as sparse vectors or low-rank matrices. 
The underlying theory, which became known under the name compressive sensing \citep{CT06,donoho,foucart2013compressed}, establishes conditions under which such data can be uniquely and efficiently reconstructed. Nonetheless, under noise corruption, most existing methods require tuning hyper-parameters with respect to the (in principle unknown) noise level. It is still an intense topic of research to understand such tuning procedures \citep{berk2021best}. 
As already mentioned by \cite{vaskevicius2019implicit,chou2021more}, the implicit sparsity regularization creates a bridge between the recent studies on gradient descent and compressive sensing. In fact, the overparametrized gradient descent provides a tuning-free alternative to established algorithms like LASSO and Basis Pursuit.
In a similar manner, our present contribution stems from the insight that in most of the above papers \citep{arora2018optimization,arora2019implicit,chou2020implicit,chou2021more}, the signs of components do not change over time when gradient flow is applied. Instead of viewing this feature as an obstacle \cite[e.g.,][Section 2]{chou2021more}, we use it to naturally link the implicit bias of gradient descent/flow to another ubiquitous problem of numerical mathematics, namely \eqref{eq:NNLS}. 

%%%%%%%%%%%%%%%%%%%%%%%%%%%
\subsection{Related work --- NNLS}
\label{sec:NNLSrelatedwork}
%%%%%%%%%%%%%%%%%%%%%%%%%%%

The first algorithm proposed to solve \eqref{eq:NNLS} appeared in 1974 in the book of \cite{lawson1995solving}, Chapter 23, where its finite convergence was proved, and a Fortran routine was presented.\footnote{This algorithm is the standard one implemented in many languages: \emph{optimize.nnls} in the SciPy package, \emph{nnls} in R, \emph{lsqnonneq} in MATLAB and \emph{nnls.jl} in Julia.} Like the previous papers \citep{stoer1971numerical, golub_saunders_70, gill1973numerically} it builds upon the solution of linear systems. Similar to the simplex method, the algorithm is an active-set algorithm that iteratively sets parts of the variables to zero in an attempt to identify the active constraints and solves the unconstrained least squares sub-problem for this active set of constraints. It is still, arguably, the most famous method for solving \eqref{eq:NNLS}, and several improvements have been proposed in a series of follow-up papers \citep{bro1997fast}, \citep{van2004fast}, \citep{myre2017tnt}, \citep{luo2011efficient}, \citep{dessole2021lawson}. Its caveat, however, is that it depends on the normal equations, which makes it infeasible for ill-conditioned or large-scale problems. Moreover, up to this point, there exist no better theoretical guarantees for the algorithm and its modifications than convergence in finitely many steps \cite[Chapter 23]{lawson1995solving}, \cite[Theorem 3]{dessole2021lawson}. \\
Another line of research has been developing projected gradient methods for solving \eqref{eq:NNLS}, which come with linear convergence guarantees \citep{kim2013non}, \citep{polyak2015projected} \citep{lin2007projected}. In contrast to active set and interior point methods, projected gradient methods do not require solving a linear system of equations at each step and thus scale better in high-dimensional problems. However, the success of these methods heavily depends on a good choice of step size and, as we already highlighted in Section \ref{sec:Introduction}, the projection step may render established acceleration methods for vanilla gradient descent useless.

%%%%%%%%%%%%%%%%%%%%%%%%%%%
\section{PROOF OF REMARK \ref{remark2.2}}
\label{sec:ProofRemark1}
%%%%%%%%%%%%%%%%%%%%%%%%%%%

Since the problem {in} \eqref{eq:NNLS} is convex, the Karush-Kuhn-Tucker (KKT) conditions are necessary and sufficient for optimality \citep[e.g.,][]{bjorck1996numerical}. 

\begin{theorem}{(Karush-Kuhn-Tucker conditions for NNLS)} \label{thm:KKT_NNLS}
A point $\x_+ \in \mathbb{R}^N$ is a solution of problem \eqref{eq:NNLS} if and only if there exists $\pmb{\mu}^* \in \mathbb{R}^n$ such that
\begin{align*}
    \pmb{\mu}^* = \A^T(\A \x_+ -  {\b}),\quad
    \pmb{\mu}^*\odot\x_+ = \0,\quad
    \pmb{\mu}^*,\x_+\geq \0.
\end{align*}
\end{theorem}
From these conditions, one can observe that all solutions of \eqref{eq:NNLS} can be represented by stationary points of the functional
\begin{align*}
     \Lover\big(\x^{(1)}, \dots, \x^{(L)}\big)
    := \frac{1}{2}\,\Big\| \A \big( \x^{(1)} \odot \cdots \odot \x^{(L)} \big) -\y   \Big\|_2^2
\end{align*}
in \eqref{eq:Lover}, i.e., points that satisfy, for all $\ell \in [L]$, the equation
\begin{equation}\label{estationary point}
   \nabla \LossPoly_{\x^{(\ell)}} \big(\x^{(1)}, \dots, \x^{(L)}\big)
    =\left[\A^{\T} \Big( \A \big(\x^{(1)} \odot \cdots \odot \x^{(L)}\big) -\y \Big) \right] \odot \Big( \bigodot_{k \neq \ell} \x^{(k)} \Big) = \0.
\end{equation}
Indeed, for any given optimal point $\mathbf{x}_+ $ of the \eqref{eq:NNLS} problem, it is straightforward to check that the conditions in Theorem \ref{thm:KKT_NNLS} imply that $\x^{(1)} = \cdots = \x^{(L)} = \x_+^{\odot \frac{1}{L}}$ is a stationary point of \eqref{eq:Lover} with $\x^{(1)} \odot \cdots \odot \x^{(L)} = \x_+$. The same argument holds for stationary points of the reduced functional \eqref{eq:LossPoly}. In particular, this implies that any solution of \eqref{eq:NNLS} can be described as the limit of gradient flow on \eqref{eq:Lover} under suitably chosen \emph{identical} initialization.

%%%%%%%%%%%%%%%%%%%%%%%%%%%
\section{PROOFS OF THEOREMS \ref{theorem:L1_equivalence_positive} AND \ref{theorem:L1_equivalence_positive_PartII}}
\label{sec:MainProofs}
%%%%%%%%%%%%%%%%%%%%%%%%%%%

Let us for convenience recall the statements of Theorems \ref{theorem:L1_equivalence_positive} and \ref{theorem:L1_equivalence_positive_PartII} before starting the proof. For simplicity, we merge both results into a single statement. 

\begin{theorem} \label{theorem:L1_equivalence_positive_merged}
    Let $\L\geq 2$, $\A\in\mathbb{R}^{\M\times\N}$ and $\y\in\mathbb{R}^{\M}$. Let $\x_0 > \0$ be fixed and let $\x(\t)$ follow the flow $\x'(\t) = -\nabla \LossPoly(\x(\t))$ with $\x(0) = \x_0$.
    Let $S_+$ be the set defined in \eqref{eq:NNLS} and set $\xprod = \x^{\odot L}$.
    
    Then the limit $\xprodinfty:= \lim_{\t\to\infty} \xprod(\t)$ exists and lies in $S_+$. Also, there exists an absolute constant $C > 0$ that only depends on the choice of $\A$, $\y$, and $\x_0$ such that 
    \begin{align*}
        {\LossBasic(\xprod(\t)) - \LossBasic(\x_+) \le \frac{C}{\t}},
    \end{align*}
    for any $\t > 0$. Let $\epsilon>0$ and assume in addition that $\x_0 = \alpha \bo$. If
    \begin{align}\label{eq:alpha_epsilon_bound_merged}
        \alpha
        \leq 
        h(Q_+,\epsilon)
        :=\begin{cases}
        \exp\left(-\frac{1}{2} - \frac{Q_+^2 + \N e^{-1}}{2\epsilon}\right) &\text{if }\L=2\\
        \left(\frac{2\epsilon}{\L(Q_+ +\N + \epsilon)}\right)^{\frac{1}{\L-2}}  &\text{if }\L>2
        \end{cases},
    \end{align}
    where $Q_+ = \min_{\z \in S_+} \|\z\|_1$, then the $\ell_1$-norm of $\xprodinfty$ satisfies
    \begin{align*}
        \|\xprodinfty\|_1 - \min_{\z \in S_+} \|\z\|_1 \leq \epsilon.
    \end{align*}
\end{theorem}

The proof consists of two major steps: First, we prove that the flow on \eqref{eq:LossPoly} converges and characterize its limits as the minimizer of a specific optimization problem. Second, we show that if $\x_0 = \alpha \bo$, the limit approximately minimizes the $\ell_1$-norm among all possible solutions of \eqref{eq:NNLS}. Whereas the second step, {including the non-asymptotic dependence between $\alpha$ and $\varepsilon$}, is taken unchanged from \cite{chou2021more}, the first step, which is the backbone of the argument, requires a different reasoning due to the (possible) non-existence of solutions $\A\z = \y$. We will rely on the concept of Bregman divergence.

\begin{definition}[Bregman Divergence]\label{def:Bregman_Divergence}
Let $\F:\Omega\to\mathbb{R}$ be a continuously-differentiable, strictly convex function defined on a closed convex set $\Omega$. The Bregman divergence associated with $\F$ for points $p,q\in\Omega$ is defined as
\begin{equation}
    D_{\F}(p,q) = \F(p) - \F(q) - \langle \nabla \F(q), p-q \rangle.
\end{equation}
\end{definition}

\begin{lemma}[\citep{Bregman1967}]
   The Bregman divergence $D_\F$ is non-negative and strictly convex in $p$.
\end{lemma}

Theorem \ref{theorem:Bregman_positive} below, which shows convergence of $\x^{\odot \L}$ and characterizes the limit $\xprodinfty$, resembles Theorem 5 by \cite{chou2021more}. Note, however, that the definition of $S_+$ is different and that Theorem \ref{theorem:Bregman_positive} does not require the existence of a solution of $\A\x=\y$. Moreover, we consider a function $\F: \mathbb{R}_{\geq 0}^{\N} \to \mathbb{R}$ that differs slightly from its counterpart in Eq.\ (2.7) of \cite{chou2021more} and is based on an important concept that appears in thermodynamics, namely, the so-called \emph{Tsallis q-logarithm}, introduced in a seminal paper  by \cite{tsallis1988possible}. It is given by 
\begin{equation}\label{eq:Bregman2}
    \F(\x) 
    = \langle \x, \ln_q(\x) \rangle,
\end{equation}
where $q = 2 - \frac{2}{L} \in [1,2)$ and the $q$-logarithm is defined as
\begin{equation*}
\ln_q(u) =
\begin{cases}
\frac{x^{1-q}-1}{1-q}, \text{ if } q\ne 1;\\
\ln(x), \text{ if } q=1,
\end{cases}
\end{equation*}
where we set $0 \log(0) = \lim_{z \to 0} z \log(z) = 0$. The $q$-logarithm $\ln_q(u)$ has the following basic properties:
\begin{itemize}
    \item[(i)] $\ln_q(1/u) = -u^{q-1}\ln_q(u)$;
    \item[(ii)] $\ln_q(uv) = \ln_q(u)+\ln_q(v) + (1-q)\ln_q(u) \ln_q(v)$
    \item[(iii)] $\ln_q(u/v) = v^{q-1}(\ln_q(u) - \ln_q(v))$
\end{itemize}

Furthermore, its derivative is given by $\frac{d}{du}(\ln_q(u))=u^{-q}$. Using these properties, it follows that the gradient of $F$ is 
\begin{align}
    \label{eq:GradF}
    \nabla F(\x) = \ln_q(\x) + \x^{\odot 1-q}
\end{align}
and that the Hessian of $F$ is  
\begin{align}
    \label{eq:HessF}
    \nabla^2F(\x) = (2-q)\diag(\x^{\odot -q}).
\end{align} 
In particular, the function $F$ is convex, since $q=2 - 2/L <2$ and, therefore, $\nabla^2F(\x) \succcurlyeq 0$. Thus, the Bregman divergence $D_{\F}$ is well-defined.
Note that $F$ can be written as 
\begin{align}\label{eq:Breman_positive}
    \F(\x) 
    = \langle \x, \ln_q(\x)\rangle=
    \begin{cases}
    \frac{1}{2} \langle  \x \log \x, \bo \rangle  & \text{if }\L = 2\\
    \frac{L}{2-L} \langle \x^{\odot \frac{2}{\L}} - \x, \bo \rangle  & \text{if }\L > 2,
    \end{cases}
\end{align}
which is worthwhile to mention since it differs from $F$ in Eq.\ (2.7) of \cite{chou2021more} only by affine shifts, i.e., both variants of $F$ induce the same Bregman divergence $D_F$ {and the following observation can directly be transferred.

\begin{lemma}[Lemma 7 of \cite{chou2021more}] \label{lem:Boundedness}
   Let $\F$ be the function defined in \eqref{eq:Breman_positive}, $\z \ge \0$ be fixed, and $\x(\t) \colon \R_{\ge 0} \to \R_{\ge 0}^\N$ be a continuous function with $\x(0) > \0$ and $\| \x(\t) \|_2 \to \infty$. Then $D_\F(\z,\x(\t)) \to \infty$.
\end{lemma}

We are now ready to state Theorem \ref{theorem:Bregman_positive}.
}

\begin{theorem}\label{theorem:Bregman_positive}
Let $\xprod(\t) = \x(\t)^{\odot\L}$ and
\begin{align} \label{eq:dynamicsBregman_positive}
    \x'(\t)
    =-\nabla \LossPoly (\x(\t))
    =-\left[\A^{\T}( \A\x^{\odot \L}(\t) -\y)\right] \odot \x^{\odot \L-1}(\t)
\end{align}
with $\x(0)\geq 0$. Then $\xprodinfty:= \lim_{\t\to\infty} \xprod(\t)$ exists and

\begin{align}\label{eq:optimal_x_positive}
		\xprodinfty \in
		&\argmin_{\z \in S_+}D_{\F}(\z,\xprod(0)),
\end{align}
where $S_+$ is defined in \eqref{eq:NNLS}. Moreover, {there exists an absolute constant $C > 0$ that only depends on the choice of $\A$, $\y$, and $\x_0$ such that,} for all $\t>0$,
\begin{align*}
    {\| \A\xprod(\t) - \b \|_2^2 - \| \A\x_+ - \b \|_2^2  \le \frac{C}{t}.}
\end{align*} 
\end{theorem}

{In the proof of Theorem \ref{theorem:Bregman_positive}, we} require the following geometric observation.

\begin{lemma} \label{lem:ConvexProjection}
	Let $\mathcal C \subset \R^n$ be a convex and closed set. Let $\w \in \mathcal C$, $\y \in \R^n$ with $\y \neq \w$, and let $\PC\y$ denote the (unique) Euclidean projection of $\y$ onto $\mathcal C$. Then
	\begin{align}
	\langle \w - \y, \w - \PC \y \rangle
	\ge \frac{1}{2}  \| \w - \PC \y \|_2^2.
	\end{align}
	Furthermore, it holds that
	\begin{align*}
	    \| \w - \PC \y \|_2^2 \le \| \w - \y \|_2^2 - \| \y - \PC \y \|_2^2.
	\end{align*}
\end{lemma}

\begin{proof}
	By definition of the projection, we have that
	\begin{align*}
	\| \w - \y \|_2^2 
	\ge \| \PC \y - \y \|_2^2
	&= \| -(\w - \PC \y) + \w - \y \|_2^2\\
	&= \| \w - \PC \y \|_2^2 - 2 \langle \w - \PC \y, \w - \y \rangle + \| \w - \y \|_2^2.
	\end{align*}
	Rearranging the terms yields
	\begin{align*}
	2 \langle \w - \y, \w - \PC \y \rangle 
	\ge \| \w - \PC \y\|_2^2
	\end{align*}
	and thus the first claim. Now note that $\langle \w - \PC \y, \y - \PC \y \rangle \le 0$ by $\PC$ being a convex projection \citep[e.g.,][Lemma 3.1]{bubeck2015convex}. This yields
	\begin{align*}
	    \| \w - \y \|_2^2 
	    &= \| \w - \PC \y - (\y- \PC \y) \|_2^2 \\
	    &= \| \w - \PC \y \|_2^2 - 2 \langle \w - \PC \y, \y - \PC \y \rangle + \| \y - \PC \y \|_2^2 \\
	    &\ge \| \w - \PC \y \|_2^2 + \| \y - \PC \y \|_2^2
	\end{align*}
	and thus the second claim.
\end{proof}

With these tools at hand, we can finally provide a proof for Theorem \ref{theorem:Bregman_positive}.

\begin{proof}[Proof of Theorem \ref{theorem:Bregman_positive}]
Let us begin with a brief outline of the proof. We will first compute the time derivative of $D_{\F}({\x_+},\xprod(\t))$, for ${\x_+} \in S_+$ where $S_+$ was defined in \eqref{eq:NNLS}. These first steps appeared before in the proof of Lemma 8 by \cite{chou2021more} and are only detailed for the reader's convenience. With the help of Lemma \ref{lem:ConvexProjection} we can then show that $\A\xprod(\t) \to \y_+$ {where $\b_+$ has been defined in \eqref{eq:CPlus}}. This part of the proof requires more effort than in existing works \citep[e.g.,][]{chou2021more} due to the possible non-existence of pre-images for $\y$ under $\A$. We wrap up (i.e., deduce boundedness and the convergence of $\| \xprod(\t) \|_2$ via Lemma \ref{lem:Boundedness}, and characterize the limit $\xprodinfty$) by following the steps of \cite{chou2021more}. Keep in mind that the definition of the set $S_+$ is different from the one used by \cite{chou2021more}. 

We start with computing $\partial_\t D_{\F}(\z,\xprod(\t))$. 
For any $\z \in \R^\N$, we have
\begin{align*}
\partial_\t D_{\F}(\z,\xprod(\t))
&= \partial_\t \left[ \F(\z) - \F(\xprod(\t)) - \langle \nabla \F(\xprod(\t)), \z-\xprod(\t) \rangle \right]\\
&= 0 - \langle \nabla \F(\xprod(\t)), \xprod'(\t) \rangle 
- \langle \partial_\t \nabla \F(\xprod(\t)), \z-\xprod(\t) \rangle
+ \langle \nabla \F(\xprod(\t)), \xprod'(\t) \rangle\\
&= - \langle \partial_\t \nabla \F(\xprod(\t)), \z-\xprod(\t) \rangle.
\end{align*} {Using \eqref{eq:HessF},} we calculate that
\begin{align*}
\partial_t \nabla \F(\xprod(\t))
&= \nabla^2 F (\xprod(\t)) \cdot \xprod'(\t)
= \frac{2}{L} \xprod(\t)^{\odot(-2+\frac{2}{L})} \odot \xprod'(\t)
= \frac{2}{L} \x(\t)^{\odot(-2L+2)} \odot \big( L\x(\t)^{\odot(L-1)} \odot \x'(\t) \big) \\
&= 2\x(\t)^{\odot(-2L+2)} \odot  \x(\t)^{\odot(L-1)} \odot \big( -\left[\A^{\T}( \A\x^{\odot \L}(\t) -\y)\right] \odot \x^{\odot \L-1}(\t) \big) \\
&= -2\left[\A^{\T}( \A\xprod(\t) -\y)\right]
\end{align*}

By definition, it is clear that $S_+ \neq \emptyset$. Let ${\x_+}$ be any element of $S_+$. Since $C_+ = \{ \A\z \colon \z \in \R_{\ge 0} \} \subset \R^\M$ is a closed convex set and $\y_+ = \mathbb P_{C_+} \y$, we can set $\mathcal C = C_+$, $\w = \A\xprod(\t) \in C_+$, and $\y = \y$ in Lemma \ref{lem:ConvexProjection} and obtain that
\begin{align} 
\begin{split} \label{eq:dtDF}
\partial_\t D_{\F}({\x_+},\xprod(\t))
&= -2\langle \A\xprod(\t) - \y, \A\xprod(\t) - \A{\x_+} \rangle \\
&\le -  \| \A\xprod(\t) - \y_+ \|_2^2.
\end{split}
\end{align}
Note that $D_\F({\x_+},\xprod(\t))$ must converge for $\t \to \infty$ as it is non-negative and non-increasing. In particular, $0 \le D_\F({\x_+},\xprod(\t)) \le D_\F({\x_+},\xprod(0))$ for all $\t \ge 0$.

We now show that $\A\xprod(\t) \to \y_+$ with the given rate. {To this end, we define the Lyapunov function
\begin{equation*}
\mathcal{E}(t) = \frac{t}{2}(\|\A\xprod(t)-\b\|_2^2 - \|\b_+ -\b\|_2^2) + \frac{1}{2}D_F(\x_+,\xprod(t)).    
\end{equation*}
The derivative $\mathcal{E}'(t)$ is given by
\begin{align*}
\mathcal{E}'(t) &= \frac{1}{2}\Big(\|\A\xprod(t)-\b\|_2^2 - \|\b-\b_+\|_2^2\Big) + \frac{t}{2} \partial_t \|\A\xprod(t)-\b\|_2^2 + \frac{1}{2} \partial_t D_F(\x_+,\xprod(t))\\
&= \ell(\xprod(t)) - \ell(\x_+) +  Lt\langle \nabla \LossPoly(\x(t)),\x'(t)\rangle  - \langle \A^T(\A\xprod(t) - \b), \x_+ - \xprod(t) \rangle\\
&= \LossBasic(\xprod(t)) - \LossBasic(\x_+) - Lt\|\nabla \LossPoly(\x(t))\|_2^2 + \langle \nabla \LossBasic(\xprod(t)),\x_+ - \xprod(t)\rangle\\
&= -D_\ell(\x_+,\xprod(t)) - Lt\|\nabla \LossPoly(\x(t))\|_2^2 \le 0.
\end{align*}
Hence, the function $\mathcal{E}(t)$ is non-increasing, which yields
\begin{align*}
t\big(\|\A\xprod(t)-\b\|_2^2 - \|\b-\b_+\|_2^2\big) \le 2\mathcal{E}(t) \le 2\mathcal{E}(0) = D_F(\x_+,\xprod(0)).
\end{align*}
This implies that $\|\A\xprod(t)-\b\|_2^2$ converges to $\|\b-\b_+\|_2^2$ with order $\mathcal{O}(1/t)$.}
In particular, we conclude that $\lim_{\t \to \infty} \A\xprod(\t) = \y_+$ {since $\b_+$ is the unique projection of $\b$ onto $C_+$ as defined in \eqref{eq:CPlus}}.

Using Lemma \ref{lem:Boundedness}, boundedness of $\| \xprod(\t) \|_2$ follows as in the work of \cite{chou2021more}. Just fix any ${\x_+} \in S_+$ and notice that $D_\F({\x_+},\xprod(\t)) \to \infty$ if $\| \xprod(\t) \|_2 \to \infty$ which contradicts the fact that $0 \le D_\F({\x_+},\xprod(\t)) \le D_\F({\x_+},\xprod(0))$. Let us denote by $B$ a sufficiently large compact ball around the origin such that $B \cap S_+ \neq \emptyset$ and $\xprod(\t) \in B$, for all $\t \ge 0$.

Now assume that there exists no ${\x_+} \in S_+ \cap B$ such that $\lim_{\t\to\infty}D_\F({\x_+},\xprod(\t)) = 0$, i.e., by compactness of $S_+ \cap B$ there exists $\varepsilon > 0$ such that $\lim_{\t\to\infty} D_\F({\x_+},\xprod(\t)) > \epsilon$, for all ${\x_+} \in S_+ \cap B$. By strict convexity of $D_\F(\cdot,\xprod(\t))$, this implies that $\xprod$ is bounded away from the set $S_+$ on $B$ and $\| \A\xprod(\t) - \y_+ \|_2$ cannot converge to zero contradicting the just obtained convergence $\lim_{\t \to \infty} \A\xprod(\t) = \y_+$. Hence, there exists ${\x_+} \in S_+$ with $\lim_{\t\to\infty} D_\F({\x_+},\xprod(\t)) = 0$.

For any such ${\x_+}$, let us assume that $\xprod(\t) \nrightarrow {\x_+}$. Then there exists $\varepsilon > 0$ and a sequence of times $\t_0,\t_1,\dots$ with $\| {\x_+} - \xprod(\t_k) \|_2 \ge \varepsilon$ and $\lim_{k\to\infty}D_\F({\x_+},\xprod(\t_k)) = 0$. Since $\xprod(\t_k)$ is a bounded sequence, a (not relabeled) subsequence converges to some $\bar{\x}$ with $\| {\x_+} - \bar{\x} \|_2 \ge \varepsilon$ and $D_\F({\x_+},\bar{\x}) = 0$. Since $D_\F(\bar{\x},\bar{\x}) = 0$ and $D_\F$ is non-negative, this is a contradiction to the strict convexity of $D_\F(\cdot,\bar{\x})$. Hence, $\xprod_\infty = \lim_{\t \to \infty} \xprod(\t) \in S_+$ exists and is the unique solution satisfying $\lim_{\t\to\infty}~D_\F~(\xprod_\infty,\xprod(\t))~=~0$.

Because $\partial_\t D_{\F}({\x_+},\xprod(\t))$ is identical for all ${\x_+} \in S_+$ (the second line of \eqref{eq:dtDF} does not depend on the choice of ${\x_+}$), the difference
\begin{align}
\Delta_{{\x_+}} = D_{\F}({\x_+},\xprod(0)) - D_{\F}({\x_+},\xprodinfty)
\end{align}
is also identical for all ${\x_+} \in S_+$. By non-negativity of $D_{\F}$,
\begin{align}
\label{eq:RandomLabel}
D_{\F}({\x_+},\xprod(0))
\geq \Delta_{{\x_+}}
= \Delta_{\xprodinfty}
= D_{\F}(\xprodinfty,\xprod(0)).
\end{align}
Thus
\begin{equation}
\begin{aligned}
\xprodinfty \in
&\argmin_{\z \in S_+}D_{\F}(\z,\xprod(0)).
\end{aligned}
\label{eq:optimal_x_positive_inproof}
\end{equation}
This complets the proof.
\end{proof}
%%%%%%%%%%%%%%%%%%%%%%%%%

Having Theorem \ref{theorem:Bregman_positive} at hand, it is straight-forward to derive Theorem \ref{theorem:L1_equivalence_positive_merged}.

\begin{proof}[Proof of Theorem \ref{theorem:L1_equivalence_positive_merged}]
    By Theorem \ref{theorem:Bregman_positive}, the limit $\x_\infty := \lim_{\t\to\infty} \x(\t)$ exists and $\xprodinfty = \x_\infty^{\odot \L}$ lies in $S_+$. 
    The quantitative bound in \eqref{eq:alpha_epsilon_bound_merged} can be deduced from \eqref{eq:optimal_x_positive} by following the proof of Theorem 4 by \cite{chou2021more} since $D_F$ based on $F$ in \eqref{eq:Breman_positive} and $D_F$ of \cite{chou2021more} are the same, cf.\ our discussion below \eqref{eq:Breman_positive}.
\end{proof}

%%%%%%%%%%%%%%%%%%%%%%%%%%%
\section{RIEMANNIAN METRIC AND RIEMANNIAN GRADIENT}
\label{sec:Riemannian}
%%%%%%%%%%%%%%%%%%%%%%%%%%%

The reparametrized flow we consider in this work induces a Riemannian metric {on the manifold $\R_{>0}$} and our method can be regarded as a Riemannian gradient approach. From this point of view, the underlying metric distorts the space and ensures that positive solutions are always obtained. Indeed, by definition of the flow in Theorem \ref{theorem:L1_equivalence_positive} we have that 
\begin{equation*}
    \x'(\t) = -\x^{\odot \L-1 } \odot[\A^\top(\A\x^{\odot \L} -\y)].
\end{equation*}
From this we can deduce the dynamics governing $\xprod=\x^{\odot \L}$ via
\begin{align*}
    \xprod'(\t)
    &= \L\x^{\odot (\L-1) }\odot\x'(\t)
    = -\L\x^{\odot (\L-1) }\odot[\x^{\odot(\L-1)}\odot[\A^\top(\A\x^{\odot \L}-\y)]]\\
    &= -\L\x^{\odot (2\L-2)}\odot[\A^\top(\A\x^{\odot \L}-\y)]
    = -\L\xprod^{\odot(2-\frac{2}{\L})}\odot[\A^\top(\A\xprod-\y)]\\
    &= -\L\xprod^{\odot(2-\frac{2}{\L})}\odot \nabla \LossBasic(\xprod)
    = -\L\xprod(t)^{\odot q}\odot \nabla \LossBasic(\xprod),
\end{align*}
where we define $q = 2 - \frac{2}{\L} \in [1,2)$. {The} matrix 
\begin{equation*}
    G(\xprod)=\diag(\xprod)^{\tfrac{2}{\L}-2} = \diag(\xprod)^{-q},
\end{equation*}
defines the Riemannian metric $\langle \u,\v \rangle_{\xprod} = \langle \xprod^{\odot (-q)}\odot \u,\v\rangle$ {on $\R_{>0}$} such that the dynamics of $\xprod'(\t)$ can be recast as
\begin{equation}\label{riemman_metric}
    \xprod'(\t) = - \L G(\xprod)^{-1}\nabla \LossBasic(\xprod) =: - \L\nablaRiemann \LossBasic(\xprod),    
\end{equation}
where $\nablaRiemann$ denotes the Riemannian gradient. For further details on Riemannian metrics and gradients, we refer the reader to Chapter 3 in the book of \cite{boumal2023introduction}.

%%%%%%%%%%%%%%%%%%%%%%%%%%%
\section{PROOF OF THEOREM \ref{thm:Accelerated}}
\label{sec:acc_proof}
%%%%%%%%%%%%%%%%%%%%%%%%%%%

We now analyze the dynamics given by
\begin{equation}
\begin{split}
    \halfstep'(t) &= -\frac{t}{2} \halfstep(t)^{\odot q} \odot \nabla \LossBasic(\xprod(t)) \nonumber,\\
    \xprod'(t) &= \frac{2}{t}{(\halfstep(t) - \xprod(t))}, \\
    \halfstep(0)&=\xprod(0)=\xprod_0> \0.
\end{split}
\end{equation}

\begin{remark}
    Let us highlight to observations:
    \begin{enumerate}
        \item[(i)] We observe that the sign of $\halfstep(t)$ does not change. {In fact, if $\xi(t_0)_i=0$, for some $t_0\ge 0$ and $i \in [N]$, then we have $\xi(t_0)_i = \xi'(t_0)_i = 0$. By the Picard-Lindel\"of theorem, we know that $\xi(t)_i = 0$, for all $t \ge 0$, which contradicts $\xi(0)_i > 0$.} Thus, the ODE admits a unique solution $\halfstep(t)>\0$, for all $t\ge 0$.
        \item[(ii)] {The equivalent second order EDO is given by 
        \begin{equation*}
            \xprod''(t) + \frac{3}{t}\xprod'(t) = -\Big(\xprod(t)+\frac{t}{2}\xprod'(t)\Big)^{\odot q} \odot \nabla \LossBasic(\xprod(t)).
        \end{equation*}} 
    \end{enumerate}
\end{remark}

\begin{proof}[{Proof of Theorem \ref{thm:Accelerated}}]
Consider ${\x_+}\in S_+$. We then have that $\ell({\x_+}) = \min_{\z \in \R_{\ge 0}^\N} \ell(\z) = \frac{1}{2}\|\y_+-\y\|_2^2$ {where $\b_+$ has been defined in \eqref{eq:CPlus}}. Now, consider the Lyapunov functional given by
\begin{equation*}
    \mathcal{E}(t)= t^2 (\LossBasic(\xprod(t))-\LossBasic({\x_+})) + \frac{4}{2-q} D_F({\x_+},\halfstep(t))
\end{equation*}
{where the function $\F(\x) = \langle \x, \ln_q(\x)\rangle$ has been defined in \eqref{eq:Breman_positive} and the Bregman divergence $D_F$ has been defined in Definition \ref{def:Bregman_Divergence}.}
We start by proving that {$\mathcal{E}(t)$ decreases in $\t$}, i.e., $\mathcal{E}'(t)\le 0$. First we will calculate the derivative $\frac{d}{dt}D_F({\x_+},\halfstep(t))$. Remember from {\eqref{eq:GradF} and \eqref{eq:HessF}} that the gradient of $F$ is given by $\nabla F(\halfstep) = \ln_q(\halfstep) + \halfstep^{\odot 1-q}$ and the Hessian is described by $\nabla^2F(\halfstep) = (2-q)\diag(\halfstep^{\odot -q})$. Using {this together with} $\halfstep'(t)=-\frac{t}{2}\halfstep(t)^{\odot q} \odot  \nabla \LossBasic(\xprod)$, we have that
\begin{align*}
\partial_t \nabla F(\halfstep(t)) &= \nabla^2 F(\halfstep(t))\halfstep'(t) = (2-q)\diag(\halfstep(t)^{\odot -q})\halfstep'(t)\\ 
&= -\frac{t}{2}(2-q)\nabla \LossBasic(\xprod(t)),
\end{align*}
{which further yields that}
\begin{align}
\frac{d}{dt}D_F({\x_+},\halfstep(t)) 
\nonumber &= \frac{d}{dt}\Big(F({\x_+}) - F(\halfstep(t)) - \langle \nabla F(\halfstep(t)),{\x_+}-\halfstep(t)\rangle\Big)\\
\nonumber& = - \langle \nabla F(\halfstep(t)),\halfstep'(t)\rangle - \langle \nabla F(\halfstep(t)),-\halfstep'(t)\rangle - \langle \partial_t \nabla F(\halfstep(t)), {\x_+}-\halfstep(t)\rangle\\
\nonumber&= - \langle \partial_t \nabla F(\halfstep(t)), {\x_+}-\halfstep(t)\rangle\\
\label{eq:pithagorean-agd}& = \frac{t}{2}(2-q)\langle \nabla \LossBasic(\xprod(t)),{\x_+}-\halfstep(t)\rangle.
\end{align}

Thus, 
\begin{align*}
\mathcal{E}'(t) &= 2t(\LossBasic(\xprod(t))-\LossBasic({\x_+})) + t^2 \langle \nabla \LossBasic(\xprod),\xprod'(t)\rangle + \frac{4}{2-q}\frac{d}{dt}D_F({\x_+},\halfstep(t)).\\
&= 2t(\LossBasic(\xprod(t))-\LossBasic({\x_+})) + t^2 \langle \nabla \LossBasic(\xprod),\xprod'(t)\rangle + 2t \langle \nabla \LossBasic(\xprod(t)),{\x_+}-\halfstep(t)\rangle \\
&=  2t(\LossBasic(\xprod(t))-\LossBasic({\x_+})) + t^2 \langle \nabla \LossBasic(\xprod),\xprod'(t)\rangle + 2t \langle \nabla \LossBasic(\xprod(t)),{\x_+} - \xprod(t) - \frac{t}{2}\xprod'(t)\rangle\\
&= 2t(\LossBasic(\xprod(t))-\LossBasic({\x_+})) + 2t \langle \nabla \LossBasic(\xprod(t)),{\x_+} - \xprod(t)\rangle\\
&= -2t D_{\LossBasic}({\x_+},\xprod(t))\le 0. \end{align*}

We can finally establish the accelerated convergence rate for the dynamics given by~\eqref{accelerated-gd-continuous}. Indeed, {since $\mathcal{E}(t)$ is decreasing, we conclude that}
    \begin{align*}
    t^2(\LossBasic(\xprod(t))-\LossBasic({\x_+})) \le \mathcal{E}(t) \le \mathcal{E}(0) = \frac{4}{2-q}D_F({\x_+},\xprod_0)
\end{align*}
{which implies that 
\begin{align*}
   0\le \LossBasic(\xprod(t))-\LossBasic({\x_+}) \le \frac{4D_F({\x_+},\xprod_0) }{(2-q)t^2}.
\end{align*}
}

\end{proof}

%%%%%%%%%%%%%%%%%%%%%%%%%%%
\section{DISCRETE ANALYSIS: PROOF OF THEOREM \ref{thm:GD_ConvergenceRate_Improved}}
\label{sec:gd_proof}
%%%%%%%%%%%%%%%%%%%%%%%%%%%

{Let us now establish} that gradient descent {on \eqref{eq:LossPoly}} converges to the solution of \ref{eq:NNLS} {by proving} Theorem \ref{thm:GD_ConvergenceRate_Improved}. 
{Let us abbreviate $M = \|\A^T\A\| 
$, let us fix $\x_0>\0$, and consider GD {on the product $\xprod = \x^{\odot L}$ defined via}
\begin{equation}\label{eq:discrete-gd-2}
    \xprod_{k+1} = \xprod_k - \etagd_k\xprod_k^{\odot q}\odot \A^T(\A\xprod_k - \y), 
\end{equation}
with $k\ge 0$, where $\etagd_k>0$ is a decreasing sequence.
We start by providing conditions under which $\ell(\xprod_k)$ forms a decreasing sequence. 
}
Recall that
\begin{equation*}
    \widetilde\nabla \LossBasic(\xprod)
    = \xprod^{\odot q}\odot \nabla \LossBasic(\xprod)
    = \xprod^{\odot q}\odot \A^T(\A\xprod - \y)
\end{equation*}
is the Riemannian gradient vector, with respect to Riemannian metric as described in {Section \ref{sec:Riemannian}}. Thus, the reparametrized GD in \eqref{eq:discrete-gd-2} reduces to the Riemannian gradient descent algorithm
\begin{equation}\label{riemannianGD}
   \xprod_{k+1} = \xprod_k - \etagd_k\widetilde\nabla \LossBasic(\xprod_k) 
\end{equation}
without retraction step.

{
\begin{remark}
It is noteworthy that in \eqref{riemannianGD} we work with a GD algorithm defined on the product representation $\xprod$. Wouldn't it be more natural to analyze GD on $\x$ instead? Fixing $\x_0>0$ this would mean to define
\begin{equation*}
\x_{k+1} = \x_k - \etagd_k \nabla\LossPoly(\x_k),
\end{equation*} 
where $\etagd_k>0$ is a decreasing sequence. In contrast to the continuous case, estimating a convergence order for this algorithm is difficult, given the non-convex nature of $\LossPoly(\x)$. Moreover, the Hessian $\nabla^2 \LossPoly(\x)$ does not have a uniform bound given its nonlinear dependency on $\x$. Taking the $L$-th Hadamard power on both sides leads to
\begin{equation*}
\xprod_{k+1} = (\x_k - \etagd_k \nabla\LossPoly(\x_k))^{\odot L}
\end{equation*}
which can be linearized by a first-order Taylor approximation of $f(\mathbf h) = (\x_k + \mathbf h)^{\odot L}$ at $\mathbf h = \0$ to yield
\begin{align*}
    \xprod_{k+1} 
    &\approx \xprod_k - L \cdot \x_k^{\odot (L-1)} \cdot \etagd_k \nabla\LossPoly(\x_k) \\
    &= \xprod_k - L \cdot \x_k^{\odot (L-1)} \cdot \etagd_k \cdot \frac{1}{L} \x_k^{\odot (L-1)} \cdot \nabla \ell(\xprod_k) \\
    &= \xprod_k - \etagd_k \xprod_k^q \cdot \nabla \ell(\xprod_k) \\
    &= \xprod_k - \etagd_k \widetilde\nabla \ell(\xprod_k).
\end{align*}
Our method in \eqref{riemannianGD} can thus be interpreted as a linearization of the naive GD routine for $\LossPoly$. Importantly, both algorithms follow the same continuous flow $\x'(t) = -\nabla \LossPoly(\x(t))$.
\end{remark}
}

\begin{lemma} 
\label{lem:DescentStep}
    {Consider GD as defined in \eqref{riemannianGD}.} Let $C \ge 0$, let $\etagd_{k-1} \le \varepsilon := \min\{ \frac{1}{2 \| \A^T \y \|_2 C^{q-1}}, \frac{1}{{2}M C^q} \}$, and let the $(k-1)$-th iterate $\xprod_{k-1} \in {\R_{> 0}^\N}$ satisfy $\| \xprod_{k-1} \|_2 \le C$. Then, $\xprod_k \in {\R_{> 0}^\N}$ and $\LossBasic(\xprod_{k})\le \LossBasic(\xprod_{k-1})$.
\end{lemma}

\begin{proof}
    To see the first claim, note that 
    \begin{align*}
        \xprod_{k} &= \xprod_{k-1} - \etagd_{k-1} \xprod_{k-1}^{\odot q}\odot \A^T(\A\xprod_{k-1} - \y) \\
        &= \xprod_{k-1} \odot \big( \boldsymbol{1} - \etagd_{k-1} \xprod_{k-1}^{\odot q-1}\odot \A^T(\A\xprod_{k-1} - \y) \big) \\
        &> 0
    \end{align*}
    by the choice of $\varepsilon$ and the assumption that $\| \xprod_{k-1} \|_2 \le C$.
    
    Let us now turn to the second claim. Since the Hessian matrix {of $\ell$ satisfies $\nabla^2 \LossBasic(\x) = \A^T\A$, Taylor's formula yields} 
    \begin{align*}
    \LossBasic(\x+{\bf{d}}) \le \LossBasic(\x) + \langle \nabla \LossBasic(\x),{\bf{d}}\rangle +  \frac{M\|{\bf{d}}\|^2}{2},  
    \end{align*}
    for all $\x,{\bf d} \in \mathbb R^N$. Applying this to $\x=\xprod_{k-1}$ and ${\bf d} = -\etagd_{k-1}\widetilde \nabla \LossBasic(\xprod_{k-1})= -\etagd_{k-1} \xprod_{k-1}^{\odot q}\odot\nabla \LossBasic(\xprod_{k-1})$, we obtain
    \begin{align*}
    \LossBasic(\xprod_k) &\le \LossBasic(\xprod_{k-1}) - \etagd_{k-1}\langle \nabla \LossBasic(\xprod_{k-1}),\widetilde\nabla \LossBasic(\xprod_{k-1})\rangle + \frac{M\etagd_{k-1}^2}{2}\|\widetilde{\nabla} \LossBasic(\xprod_{k-1})\|^2\\
    &= \LossBasic(\xprod_{k-1}) - \etagd_{k-1} \left\langle \widetilde\nabla \LossBasic(\xprod_{k-1}),\left( \1 - \frac{M\etagd_{k-1}}{2}\xprod_{k-1}^{\odot q} \right) \odot \nabla \LossBasic(\xprod_{k-1}) \right\rangle\\
    &= \LossBasic(\xprod_{k-1}) - \etagd_{k-1} \left\langle \nabla \LossBasic(\xprod_{k-1}), \xprod_{k-1}^{\odot q}\odot \left( \1 - \frac{M\etagd_{k-1}}{2}\xprod_{k-1}^{\odot q} \right) \odot \nabla \LossBasic(\xprod_{k-1}) \right\rangle. 
    \end{align*}
    By assumption we have that {$\xprod_{k-1}^{\odot q} > 0$ and} $(\1 - \frac{M}{2}\etagd_{k-1}\xprod_{k-1}^{\odot q})\ge 0$ which implies that $\LossBasic(\xprod_k)\le \LossBasic(\xprod_{k-1})$ and thus the claim.
\end{proof}

{Second, we show that the \emph{Bregman divergence balls} derived from the Tsallis q-logarithm $F$ {defined in \eqref{eq:Breman_positive}} are bounded.} 

\begin{lemma}
\label{lem:BoundedDFball}
    Let $F(\z) = \langle \z,\ln_q(\z) \rangle$ as in \eqref{eq:Breman_positive}, for $q \in [1,2)$. Then, for any $C\ge 0$ and $\x \in \R_{\ge 0}^\N$, the set
    \begin{align*}
        B_{D_F,C,\x} := \{ \z \in \R_{\ge 0}^\N \colon D_F(\x,\z) \le C \}
    \end{align*}
    is bounded. In particular, $B_{D_F,C,\x} \subset R B_\infty$ where $B_\infty$ is the $\ell_\infty$ unit ball and
    \begin{equation}
        R = \begin{cases}
            2\max(\|\x\|_\infty,C+\ln(2)\|\x\|_\infty) &\text{if }q=1\\
            \big(C + \frac{1}{q-1}\|\x\|_\infty^{2-q}\big)^\frac{1}{2-q} &\text{if }q\in(1,2).
        \end{cases}
    \end{equation}
\end{lemma}

\begin{proof}
    To show the claim, we argue that unbounded sequences are not contained in $B_{D_F,C,\x}$. 
    First note that 
    $$D_F(\x,\z) = \sum_{i = 1}^\N D_f(x_i,z_i),$$
    where $f(z) = z \ln_q(z)$ is a univariate marginal of $F$ and $D_f(x,z) \ge 0$ by convexity of $f$.
    
    Let now $\z_k \subset \R_{\ge 0}^\N$ be a sequence with $\| \z_k \|_\infty \to \infty$. Then, there exists $i\in [\N]$ and a sub-sequence $k_j$ such that $\lim_{j\to\infty}(z_{k_j})_i = \infty$. Hence, we obtain that
    \begin{align*}
        D_F(\x,\z_k) 
        &\geq D_f(x_i,(z_k)_i)\\
        &= x_i \ln_q(x_i) 
        - (z_k)_i \ln_q( (z_k)_i ) 
        - ( \ln_q( (z_k)_i ) + (z_k)_i^{1-q} ) (x_i - (z_k)_i ) \\
        &= x_i \ln_q(x_i) 
        - x_i ( \ln_q( (z_k)_i ) + (z_k)_i^{1-q} )
        + (z_k)_i^{2-q} \\
        &= x_i \ln_q(x_i) 
        + (z_k)_i^{2-q} \Big( 1 - \frac{x_i}{(z_k)_i} - \frac{x_i \ln_q( (z_k)_i )}{(z_k)_i^{2-q}} \Big) \\
        &\overset{k \to \infty}{\to} \infty
    \end{align*}
    by assumption on $q$. We now derive the non-asymptotic bound. Since $D_F(\x,\z) = \sum_{i = 1}^\N D_f(x_i,z_i)$ and $D_f\geq0$, {the $i$-th entry $z_i$ (and with it $\| \z \|_\infty$) can be chosen largest} if $z_j=x_j$, for all $j\neq i$. Thus, we simply need to {control}
    \begin{equation*}
        \max_{z\geq 0} z \quad\text{subject to}\quad D_f(x,z)\leq C.
    \end{equation*}
    To get the bound, we separate the case $q=1$ and $q\neq1$. When $q=1$, we have
    \begin{align*}
        D_f(x,z)
        &= x\ln(x) - x  + z - \ln(z)x.
    \end{align*}
    {Taking into consideration} the constraint, we obtain
    \begin{equation}\label{eq:Bregman_1}
        z - \ln(z)x \leq C + x - x\ln(x).
    \end{equation}
    Let $f(z) := z-\ln(z)x$. Note that $f$ is convex with minimum at $z=x$. Furthermore, $f'(z)\geq \frac{1}{2}$ for all $z\geq 2x$. Therefore 
    \begin{align*}
        f(z)
        &\geq
        \begin{cases}
            f(x) &\text{if }z < 2x\\
            f(2x) + \frac{1}{2}(z-2x) &\text{if }z \geq 2x
        \end{cases}\\
        &\geq
        \begin{cases}
            x-\ln(x)x &\text{if }z < 2x\\
            x-\ln(2x)x + \frac{1}{2}z &\text{if }z \geq 2x.
        \end{cases}
    \end{align*}
    Since $z,x\geq0$, {we have that if $z\geq 2x$ then} \eqref{eq:Bregman_1} implies 
    \begin{equation*}
        x-\ln(2x)x + \frac{1}{2}z \leq C + x - x\ln(x),
    \end{equation*}
    or equivalently, $z\leq 2(C+\ln(2)x)$. {Applying a case distinction between $z\geq 2x$ and $z < 2x$,} we obtain that
    \begin{equation*}
        \{ \z \in \R_{\ge 0}^\N \colon D_F(\x,\z) \le C \}
        \subset \{ \z \in \R_{\ge 0}^\N \colon z_i \leq 2\max(x_i,C+\ln(2)x_i) \}
        \subset R B_\infty. 
    \end{equation*}
     When $q\neq1$, we have
    \begin{align*}
        D_f(x,z)
        &= x\ln_q(x) + \frac{x}{1-q} - \frac{2-q}{1-q}xz^{1-q}+z^{2-q}\\
        &= \frac{x^{2-q}}{1-q} - \frac{2-q}{1-q}xz^{1-q}+z^{2-q}.
    \end{align*}
    {Taking into consideration} the constraint $C\geq D_f(x,z)$, we obtain
    \begin{equation*}
        z^{2-q} + \frac{2-q}{q-1}xz^{1-q} \leq C + \frac{x^{2-q}}{q-1}.
    \end{equation*}
    Since $q\in(1,2)$, {we know that} $\frac{2-q}{q-1}x\geq 0$. Hence we have
    \begin{align*}
        z^{1-q} & \leq \frac{(q-1)C + x^{2-q}}{(2-q)x}\\
        z^{2-q} &\leq C + \frac{x^{2-q}}{q-1}.
    \end{align*}
    and consequently
    \begin{align*}
        \{ \z \in \R_{\ge 0}^\N \colon D_F(\x,\z) \le C \}
        \subset \big\{ \z \in \R_{\ge 0}^\N \colon \Big(\frac{(2-q)x}{(q-1)C + x^{2-q}}\Big)^\frac{1}{q-1} \leq z_i \leq \big(C + \frac{x^{2-q}}{q-1}\big)^\frac{1}{2-q} \big\}
        \subset R B_\infty.
    \end{align*}
    This completes the proof.
\end{proof}

{We have now all tools at hand to prove Theorem \ref{thm:GD_ConvergenceRate_Improved}.} Since the proof strategy is a discrete analog of the proof of Theorem \ref{thm:Accelerated}, it is convenient to read the latter before turning to the proof below.

\begin{proof}[Proof of Theorem \ref{thm:GD_ConvergenceRate_Improved}] 
    {For $\x_+ \in S_+$, consider} the energy function
    \begin{equation*}
        \mathcal{E}_k = \etagd_k k (\LossBasic(\xprod_k) - \LossBasic({\x_+})) + \frac{1}{2-q}D_F({\x_+},\xprod_k)
    \end{equation*}
    with $\etagd_k$ as defined in the proposition statement
    and let $\mathfrak{C} \ge 1$ be sufficiently large such that
    \begin{align}
    \label{eq:C_def_new}
        \{ \z \in \R_{\ge 0}^\N \colon D_F({\x_+},\z) \le D_F({\x_+},\xprod_0) + c_{\max} \} \cup \{{\x_+}\} \subset \mathfrak{C}B_2,
    \end{align}
    where $B_2$ is the $\ell_2$ unit ball. 
    By Lemma \ref{lem:BoundedDFball} such a {$\mathfrak{C} \le D_F({\x_+},\xprod_0) + c_{\max} + \| \x_+\|_\infty$} always exists and only depends on $c_{\max}$ and the distance between ${\x_+}$, $\xprod_0$, and the origin. Clearly, $\xprod_0 \in \mathfrak{C}B_2 \cap \R_{>0}^\N$. {Recall that $\etagd_k = \frac{\cgd}{(k+2)^\gamma}$, for $c > 0$ having been defined in the theorem statement and depending on $\mathfrak{C}$.}
    
    We first prove by induction that $\xprod_k \in \mathfrak{C}B_2 \cap \R_{>0}^\N$ and $\mathcal{E}_k \le \mathcal{E}_{k-1} + \cgd \left( \frac{(2-q)^{-1}}{(k+1)^{2\gamma}} - \frac{\delta}{2(k+1)^\gamma} \right) $, for any $0 \le k \le K$, where we set $\mathcal{E}_{-1} = \infty$.
    For $k=0$, the claim is trivially fulfilled.\\
    Assume now the claim holds for {$(k-1) \in [0,K)$}. Then $\xprod_{k-1} \in \R_{>0}^\N$ and $\| \xprod_{k-1} \|_2 \le \mathfrak{C}$, which gives for any $i \in [\N]$ that 
    \begin{align}
    \label{eq:EntryLB_new}
    \begin{split}
        (\xprode_k)_i 
        &= (\xprode_{k-1})_i - [\etagd_{k-1} \widetilde \nabla \LossBasic (\xprod_{k-1})]_i \\
        &= (\xprode_{k-1})_i \cdot \Big( 1 - \etagd_{k-1} (\xprode_{k-1})_i^{q-1} [\A^T(\A\xprod_{k-1} - \y)]_i \Big) \\
        &\ge \frac{1}{2} (\xprode_{k-1})_i, 
    \end{split}
    \end{align}
    where we used that, by definition of $c$ in $\etagd_{k-1}$,
    \begin{align*}
        \etagd_{k-1} (\xprode_{k-1})_i^{q-1} [\A^T(\A\xprod_{k-1} - \y)]_i 
        \le \etagd_{k-1} (\mathfrak{C}^{q} M + \mathfrak{C}^{q-1} \| \A^T \y \|_2)
        \le \frac{1}{2}.
    \end{align*}
    Consequently, $\xprod_{k} \in \R_{>0}^\N$ as well. Let us now observe that, {by \eqref{eq:GradF} and \eqref{eq:HessF},} gradient and Hessian of $\tilde\x \mapsto D_F({\x_+},\cdot)|_{\tilde\x}$ are given by 
    \begin{align*}
        \nabla D_F({\x_+},\cdot)|_{\tilde\x} &= \nabla_{\tilde\x} \Big( F({\x_+}) - F(\tilde\x) - \langle \nabla F(\tilde\x), {\x_+} - \tilde\x\rangle\Big)\\
        &= -\nabla F(\tilde\x) -  \nabla^2 F(\tilde\x) ({\x_+} - \tilde\x) + \nabla F(\tilde\x)\\
        &=-\nabla^2 F(\tilde\x)({\x_+} - \tilde\x) = -(2-q) \cdot \diag (\tilde\x^{\odot -q}) ({\x_+} - \tilde\x)\\ &= -(2-q) \cdot \tilde\x^{\odot -q}\odot ({\x_+} - \tilde\x)
    \end{align*}
    and 
    \begin{align*}
        \nabla^2 D_F({\x_+},\cdot)|_{\tilde\x}
        = -(2-q) \cdot \diag \Big( (q-1) \cdot \tilde\x^{\odot -q} - q \cdot {\x_+} \odot \tilde\x^{\odot -(q+1)} \Big).
    \end{align*}
    We thus can approximate $D_F({\x_+},\xprod_k)$ by the second-order Taylor approximation of $\tilde\x \mapsto D_F({\x_+},\cdot)|_{\tilde\x}$ around $\xprod_{k-1}$
    \begin{align}
    \label{eq:TaylorDF_new}
    \begin{split}
        &D_F({\x_+},\xprod_k) - D_F({\x_+},\xprod_{k-1}) \\
        &\le - \etagd_{k-1} \langle \nabla D_F({\x_+},\cdot)|_{\xprod_{k-1}}, \widetilde\nabla \LossBasic(\xprod_{k-1})\rangle 
        + \sup_{\tilde\x \in [\xprod_{k-1},\xprod_k]} \etagd_{k-1}^2 \Big| \widetilde\nabla \LossBasic(\xprod_{k-1})^T \cdot \nabla^2 D_F({\x_+}, \cdot)|_{\tilde\x} \cdot \widetilde\nabla \LossBasic(\xprod_{k-1}) \Big| \\
        &\le (2-q) \etagd_{k-1} \langle  {\x_+} - \xprod_{k-1}, \nabla \LossBasic(\xprod_{k-1})\rangle + \frac{\cgd}{(k+1)^{2\gamma}},
    \end{split}
    \end{align}
    where $[\xprod_{k-1},\xprod_k]$ denotes the connecting line between $\xprod_{k-1}$ and $\xprod_k$, and we used the following two observations: first,
    \begin{align}
    \label{eq:DFgrad_new}
    \begin{split}
        \langle \nabla D_F({\x_+},\cdot)|_{\xprod_{k-1}}, \widetilde\nabla \LossBasic(\xprod_{k-1})\rangle 
        &= -(2-q) \langle \xprod_{k-1}^{\odot -q}\odot ({\x_+} - \xprod_{k-1}), \xprod^{\odot q}_{k-1}\nabla \LossBasic(\xprod_{k-1})\rangle\\
        &= -(2-q) \langle  {\x_+} - \xprod_{k-1}, \nabla \LossBasic(\xprod_{k-1})\rangle;
    \end{split}
    \end{align}
    {and second,}
    \begin{align*}
        &\sup_{\tilde\x \in [\xprod_{k-1},\xprod_k]} \etagd_{k-1}^2 \Big| \widetilde\nabla \LossBasic(\xprod_{k-1})^T \cdot \nabla^2 D_F({\x_+}, \cdot)|_{\tilde\x} \cdot \widetilde\nabla \LossBasic(\xprod_{k-1}) \Big| \\
        &\le \sup_{\tilde\x \in [\xprod_{k-1},\xprod_k]} \etagd_{k-1}^2 \sum_{i=1}^\N  \Big| [\nabla^2 D_F({\x_+}, \cdot)|_{\tilde\x}]_{i,i} \Big| \cdot [\widetilde\nabla \LossBasic(\xprod_{k-1})]_i^2 \\
        &\le \sup_{\tilde\x \in [\xprod_{k-1},\xprod_k]} \etagd_{k-1}^2 \sum_{i=1}^\N (2-q) \cdot \Big( |  (q-1) \cdot \tilde x_i^{-q} | + | q \cdot (x_+)_i \cdot \tilde x_i^{-(q+1)}  | \Big) \cdot (\xprode_{k-1})_i^{2q} \cdot [\A^T(\A\xprod_{k-1} - \y)]_i^2 \\
        &\le 30 \cdot \mathfrak{C}^q \cdot \etagd_{k-1}^2 \cdot \| \A^T(\A\xprod_{k-1} - \y)\|_2^2 
        \le \etagd_{k-1}^2 \cdot 60 \cdot \mathfrak{C}^q \cdot (\mathfrak{C}^{2q} M^2 + \| \A^T \y \|_2^2) \\
        &\le \frac{\cgd}{(k+1)^{2\gamma}},
    \end{align*}
    where we used in the third from last step that $\tilde x_i \ge \frac{1}{2} (\xprode_{k-1})_i$ for any $\tilde\x \in [\xprod_{k-1},\xprod_k]$ by a similar calculation as in \eqref{eq:EntryLB_new}, that $(2-q)(q-1)$ and $(2-q)q$ are both upper bounded by $\frac{9}{4}$, and that $\| \xprod_{k-1} \|_2, \| {\x_+} \|_\infty \le \mathfrak{C}$ by our induction hypothesis and \eqref{eq:C_def_new}. The ultimate inequality then just follows from {recalling the definition of $c$ in $\etagd_{k-1}$}.\\
    Since by $\etagd_k \le \etagd_{k-1}$ and Lemma \ref{lem:DescentStep} 
    \begin{align*}
        \mathcal{E}_k - \mathcal{E}_{k-1}
        &= \etagd_k k (\LossBasic(\xprod_k) - \LossBasic({\x_+})) - \etagd_{k-1} (k-1) (\LossBasic(\xprod_{k-1}) - \LossBasic({\x_+})) + \frac{1}{2-q} \Big( D_F({\x_+},\xprod_k)
         - D_F({\x_+},\xprod_{k-1}) \Big) \\
        &\le \etagd_k (\LossBasic(\xprod_k) - \LossBasic({\x_+})) + \etagd_{k-1} (k-1) (\LossBasic(\xprod_k) - \LossBasic(\xprod_{k-1})) + \frac{1}{2-q} \Big( D_F({\x_+},\xprod_k)
         - D_F({\x_+},\xprod_{k-1}) \Big) \\
        &\le \etagd_k (\LossBasic(\xprod_k) - \LossBasic({\x_+})) + \frac{1}{2-q} \Big( D_F({\x_+},\xprod_k)
         - D_F({\x_+},\xprod_{k-1}) \Big) \\
        &\le \etagd_{k-1} (\LossBasic(\xprod_{k-1}) - \LossBasic({\x_+})) + \frac{1}{2-q} \Big( D_F({\x_+},\xprod_k)
         - D_F({\x_+},\xprod_{k-1}) \Big),
    \end{align*}
    we obtain from \eqref{eq:TaylorDF_new} that
    \begin{align*}
        \mathcal{E}_k &\le  \mathcal{E}_{k-1} 
        + \etagd_{k-1} \big(\LossBasic(\xprod_{k-1}) - \LossBasic({\x_+})\big) 
        + \etagd_{k-1} \langle  {\x_+} - \xprod_{k-1}, \nabla \LossBasic(\xprod_{k-1})\rangle + \frac{\cgd (2-q)^{-1}}{(k+1)^{2\gamma}} \\
        &= \mathcal{E}_{k-1} - \etagd_{k-1} D_\LossBasic({\x_+}, \xprod_{k-1})  + \frac{\cgd(2-q)^{-1}}{(k+1)^{2\gamma}} \\
        &= \mathcal{E}_{k-1} - \frac{\etagd_{k-1}}{2} \| \A ({\x_+} - \xprod_{k-1}) \|_2^2  + \frac{\cgd (2-q)^{-1}}{(k+1)^{2\gamma}} \\
        &\le \mathcal{E}_{k-1} + \cgd \left( \frac{(2-q)^{-1}}{(k+1)^{2\gamma}} - \frac{\delta}{2(k+1)^\gamma} \right),
    \end{align*}
    where we calculated the explicit form of $D_\LossBasic$ in the penultimate step and used that by assumption 
    \begin{align*}
        \| \A ({\x_+} - \xprod_{k-1}) \|_2^2 \ge \Big( \| \A \xprod_{k-1} - \y \|_2 - \| \A{\x_+} - \y \|_2 \Big)^2 \ge \delta,
    \end{align*}
    for $k-1 < K$. To conclude the induction step, we need to argue that $\| \xprod_k \|_2 \le \mathfrak{C}$. To this end, note that
    \begin{align}
    \begin{split}
    \label{eq:Recursion}
        D_F({\x_+},\xprod_k) &\le (2-q) \mathcal{E}_k
        \le (2-q) \mathcal{E}_0 + \cgd \sum_{\j=1}^k \Big(\frac{1}{(\j+1)^{2\gamma}} - \frac{(2-q)\delta}{2(\j+1)^\gamma} \Big) \\
        &\le D_F({\x_+},\xprod_0) + c_{\max},
    \end{split}
    \end{align}
    where we used Lemma \ref{lem:sum} and $\cgd < 1$ in the last step. Clearly, this implies $\| \xprod_k \|_2 \le \mathfrak{C}$ by \eqref{eq:C_def_new}.
    
    Using once more the recursive bound of $\mathcal{E}_k$ {appearing in \eqref{eq:Recursion}}, we estimate for any $k \le K$ that
    \begin{align*}
        \etagd_{k} k (\LossBasic(\xprod_k) - \LossBasic({\x_+})) \le \mathcal{E}_k 
        \le  \frac{1}{2-q} ( D_F({\x_+}, \xprod_0) + c_{\max} ).
    \end{align*}
    Reorganizing the estimate and using that $(k+2)^\gamma/k \le 2k^{1-\gamma}$, for $k \ge 2$, yields the claim.
\end{proof}

To conclude the proof of Theorem \ref{thm:GD_ConvergenceRate_Improved}, we finally provide the estimate of the sum in \eqref{eq:Recursion}.

\begin{lemma}
\label{lem:sum}
    Let $\gamma \in (0,1)$, $q \in [1,2)$, and $\delta > 0$. Then,
    \begin{align}
    \label{eq:sum}
        \sum_{\j=1}^k \Big(\frac{1}{(\j+1)^{2\gamma}} - \frac{(2-q)\delta}{2(\j+1)^\gamma} \Big)
        \le 
        \begin{cases}
            \zeta(2\gamma) & \text{for } \gamma > \frac{1}{2} \\
            2 \left( \log\big( \frac{2}{(2-q) \delta} \big) + (2-q)\gamma \right) & \text{for } \gamma = \frac{1}{2} \\
            \frac{1}{1-2\gamma} \Big( \frac{2}{(2-q)\delta} \Big)^{\frac{1}{\gamma}-2} + 2 (2-q)\delta & \text{for } \gamma < \frac{1}{2}
        \end{cases},
    \end{align}
    where $\zeta$ denotes the Riemann zeta function.
\end{lemma}
\begin{proof}
    For $\gamma > \frac{1}{2}$, we have that
    \begin{align*}
        \sum_{\j=1}^k \Big(\frac{1}{(\j+1)^{2\gamma}} - \frac{(2-q)\delta}{2(\j+1)^\gamma} \Big)
        \le \sum_{\j=1}^k  \frac{1}{(\j+1)^{2\gamma}}
        \le \zeta(2\gamma).
    \end{align*}
    For $\gamma \le \frac{1}{2}$, observe that
    \begin{align*}
         \sum_{\j=1}^k \frac{1}{(\j+1)^{2\gamma}}
         = \sum_{\j=2}^{k+1} \int_{\j-1}^\j \frac{1}{\j^{2\gamma}} \,d\tau
         \le \sum_{\j=2}^{k+1} \int_{\j-1}^\j \frac{1}{\tau^{2\gamma}} \,d\tau
         =
         \int_1^{k+1} \frac{1}{\tau^{2\gamma}} \,d\tau
         \le h(k+1),
    \end{align*}
    where $h(k+1) = \log(k+1)$, for $\gamma = \frac{1}{2}$, and $h(k+1) = (1-2\gamma)^{-1} (k+1)^{1-2\gamma}$ else.
    Similarly,
    \begin{align*}
         \sum_{\j=1}^k \frac{1}{(\j+1)^{\gamma}}
         = \sum_{\j=2}^{k+1} \int_{\j}^{\j+1} \frac{1}{\j^{\gamma}} \,d\tau
         \ge \int_2^{k+2} \frac{1}{\tau^{\gamma}} \,d\tau
         = (1-\gamma)^{-1} \big( (k+2)^{1-\gamma} - 2^{1-\gamma} \big).
    \end{align*}
    Thus, the sum in \eqref{eq:sum} can be bounded by the maximal value of
    \begin{align}
    \label{eq:f_bound}
        f \colon [0,\infty) \to \R,
        \qquad 
        f(t) := h(t+1) - \frac{(2-q)\delta }{2} (1-\gamma)^{-1} \Big( (t+2)^{1-\gamma} - 2^{1-\gamma} \Big).
    \end{align}
    Since
    \begin{align*}
        f(t) \le g(t) := h(t+2) - \frac{(2-q)\delta}{2} (1-\gamma)^{-1} \Big( (t+2)^{1-\gamma} - 2^{1-\gamma} \Big), 
    \end{align*}
    one has that $\max_{t \ge 0} f(t) \le \max_{t \ge 0} g(t)$, which can be determined by basic calculus to be
    \begin{align*}
        \max_{t \ge 0} g(t) &= 
        \begin{cases}
            2 \log\Big( \frac{2}{(2-q) \delta} \Big) - 2 + \sqrt{2} (2-q)\delta & \text{for } \gamma = \frac{1}{2} \\
            \frac{1}{1-2\gamma} \Big( \frac{2}{(2-q)\delta} \Big)^{\frac{1}{\gamma}-2} - \frac{(2-q)\delta }{2 (1-\gamma)} \Big( \frac{2}{(2-q)\delta} \Big)^{\frac{1}{\gamma}-1} + \frac{2^{1-\gamma} (2-q)\delta}{2 (1-\gamma)} & \text{for } \gamma < \frac{1}{2}
        \end{cases}
        \\
        &\le \begin{cases}
            2 \left( \log\big( \frac{2}{(2-q) \delta} \big) + (2-q)\delta \right) & \text{for } \gamma = \frac{1}{2} \\
            \frac{1}{1-2\gamma} \Big( \frac{2}{(2-q)\delta} \Big)^{\frac{1}{\gamma}-2} + 2 (2-q)\delta & \text{for } \gamma < \frac{1}{2}
        \end{cases}
    \end{align*}
    This completes the proof.
\end{proof}

%%%%%%%%%%%%%%%%%%%%%%%%%%%
\section{APPLICATION TO SPARSE RECOVERY}
\label{sec:NNLS_sparse}
%%%%%%%%%%%%%%%%%%%%%%%%%%%

Apart from applications that seek non-negative physical quantities, see Section \ref{sec:Introduction}, NNLS is also an attractive method for compressive sensing if the sparse ground truth is non-negative. Indeed, under certain assumptions on the measurement operator $\A$, solving \eqref{eq:NNLS} promotes the sparsity of its solution for free (no hyper-parameter tuning) and comes with improved robustness. This observation dates back to the seminal works  of \cite{donoho2005sparse,donoho2010counting}, which connect the topic to the theory of convex polytopes. (Interestingly, even before the modern theory of sparse recovery emerged in the form of compressive sensing, \cite{donoho1992maximum} exploited non-negativity to recover sparse objects.) Consequent works studied the uniqueness of positive solutions of underdetermined systems \citep{bruckstein2008uniqueness} and the generalization to low-rank solutions on the cone of positive-definite matrices \citep{wang2010unique}. Other works established conditions under which NNLS would succeed to retrieve a sparse vector even in a noisy setting \citep{kueng2017robust, shadmi2019sparse, slawski2011sparse,slawski2013non, meinshausen2013sign}. Two key concepts for such results are the \emph{null space property} \citep{foucart2013compressed}, 
and the \emph{$\mathcal{M}_{+}$ criterion} \citep{bruckstein2008uniqueness}. 

\begin{definition}[{\cite[Definition 4.17]{foucart2013compressed}}] \label{def:NSP:statement}
	\label{def_NSP} A matrix $\A \in \mathbb{R}^{\M \times \N}$ is said to satisfy the \emph{ robust null space property (NSP)} of order $s \in [N]$ with constants $0 <\rho<1$ and $\tau>0$ if for any set $S \subset [N]$ of cardinality $|S| \leq s$, it holds that $ \onenorm{\v_S} \leq \rho  \onenorm{\v_{S^c}} + \tau \|\A\v\|$, for all  $ \v \in \mathbb{C}^N$.
\end{definition}

\begin{definition}[{\citep{bruckstein2008uniqueness}}] 
Let $\A \in \mathbb{R}^{\M \times \N}$. We say $\A$ obeys the $\mathcal{M}_{+}$ criterion with vector $\u_{\mathcal{M}_{+}}$ if there exists $\u_{\mathcal{M}_{+}} \in \mathbb{R}^\M$ such that $\A^\top \u_{\mathcal{M}_{+}} >0$, i.e., if $\A$ admits a strictly-positive linear combination of its rows. 
\end{definition}
Note that the $\mathcal{M}_{+}$ criterion is a necessary condition for an underdetermined ($\M < \N$) system to admit a unique non-negative solution \cite[Theorem 5]{wang2010unique}. Examples of matrices $\A \in \R^{\M\times \N}$ satisfying the $\mathcal M_+$ criterion are (i) matrices with i.i.d.\ Bernoulli entries, (ii) matrices the columns of which can be written as independent 1-subgaussian random vectors \citep{shadmi2019sparse}, (iii) matrices the columns of which form an outwardly k-neighborly polytope \citep{donoho2005sparse}, and (iv) adjacency matrices of bipartite expander graphs \citep{wang2010unique}. A recent robustness result relying on null space property and $\mathcal M_+$ criterion is the following.

\begin{theorem}[{\cite[Theorem 4]{kueng2017robust}}]
\label{thm:Kueng}
Suppose that $\A\in \mathbb{R}^{\M \times \N}$ obeys the NSP of order $s \leq \N$ with constants $0 <\rho<1$ and $\tau>0$ and the $\mathcal{M}_{+}$ criterion with the vector $\u_{\mathcal{M}_{+}}$. Then $\A$ allows stable reconstruction of any non-negative $s$-sparse vector $\x_*$ from $\y = \A\x_* + \boldsymbol\epsilon$ via \eqref{eq:NNLS}. In particular, for any $1\leq p \leq q$, the unique solution $\x_+$ of \eqref{eq:NNLS} is guaranteed to obey
\begin{equation}
    \|\x_+ - \x_* \|_p \leq \frac{C}{s^{1-1/p}}\sigma_s(\x)_1+\frac{D}{s^{1/q-1/p}}(\| \u_{\mathcal{M}_{+}} \|_2+\tau)\|\boldsymbol\epsilon\|_2,
\end{equation}
where $C$ and $D$ only depend on $\rho$, the condition number of the diagonal matrix $\diag(\A^\top \u_{\mathcal{M}_{+}})$.
\end{theorem}

Theorem \ref{thm:Kueng}, which was later generalized to arbitrary $\ell_p$-quasinorms \cite[Theorem 2]{shadmi2019sparse}, shows that if $\A$ behaves sufficiently well, the solution of \eqref{eq:NNLS} stably reconstructs any non-negative s-sparse vector from noisy measurements at least as well as conventional programs for sparse recovery. In particular, neither sparsity regularization nor parameter tuning is required. The program only relies on the geometry imposed by its constraints. The work of \cite{shadmi2019sparse} even showed that \eqref{eq:NNLS} outperforms Basis Pursuit Denoising in retrieving a sparse solution from noisy measurements. 
However, let us mention that measurement operators $\A$ appearing in applications normally do not satisfy the assumptions of Theorem \ref{thm:Kueng}. In such scenarios, an additional sparsity regularization is still needed when working with \eqref{eq:NNLS}. 

{In our proposed approach, we get sparse regularization as a by-product of reusing existing theory in our proofs. Indeed,} we finally observe that if $\x_0$ is chosen sufficiently close to zero, the limit of {the reparametrized} gradient flow is approximately minimizing a (weighted) $\ell_1$-norm among all possible solutions of \eqref{eq:NNLS}. 
The following theorem formalizes these claims.
Let us emphasize that a small initialization is only required in Theorem~\ref{theorem:L1_equivalence_positive_PartII} to obtain additional $\ell_1$-regularization. General instances of NNLS can be solved by {our} method with arbitrary positive initialization. Further note that in order to guarantee a similar additional regularization in the case of general measurements $\A$, the established techniques for solving \eqref{eq:NNLS} --- grouped as (i)-(iii) in the discussion in {Appendix \ref{sec:RelatedWork}} --- would require notable adjustments both in methodology and in theory.

\begin{theorem}
\label{theorem:L1_equivalence_positive_PartII}
    Let $\epsilon>0$. In addition to the assumptions of Theorem \ref{theorem:L1_equivalence_positive} assume that $\x_0 = \alpha \bo$. If
    \begin{align}\label{eq:alpha_epsilon_bound}
        \alpha
        &\leq 
        h(Q_+,\epsilon) \\
        &:=\begin{cases}
        \min\left(e^{-\frac{1}{2}}, \exp\left(\frac{1}{2} - \frac{Q_+^2 + \N e^{-1}}{2\epsilon}\right) \right) &\text{if }\L=2\\
        \left(\frac{2\epsilon}{\L(Q_+ +\N + \epsilon)}\right)^{\frac{1}{\L-2}}  &\text{if }\L>2
        \end{cases}, \notag 
    \end{align}
    where $Q_+ = \min_{\z \in S_+} \|\z\|_1$, then the $\ell_1$-norm of $\xprodinfty$ satisfies
    \begin{align*}
        \|\xprodinfty\|_1 - \min_{\z \in S_+} \|\z\|_1 \leq \epsilon.
    \end{align*}

\end{theorem}

{We {have proved} Theorem \ref{theorem:L1_equivalence_positive_PartII} together with Theorem \ref{theorem:L1_equivalence_positive} in Section \ref{sec:MainProofs} {above}.} The additional $\ell_1$-regularization that is described in Theorem \ref{theorem:L1_equivalence_positive_PartII}, for small $\alpha > 0$, allows finding NNLS-solutions that are of lower complexity since there is a strong connection between small $\ell_1$-norm and effective sparsity. In particular, this allows stable reconstruction of (almost) non-negative ground truths if $\A$ is well-behaved, e.g., if $\A$ satisfies standard assumptions for sparse recovery like suitable robust null space and quotient properties \citep{foucart2013compressed}. 

\begin{definition}[{\cite[Definition 11.11]{foucart2013compressed}}]\label{def:l1_quotient} A measurement matrix $\A \in \mathbb{R}^{\M \times \N}$ is said to possess the $\ell_1$-quotient property with constant $d$ relative to the $\ell_2$-norm if, for all $\b \in \mathbb{R}^\M$, there exists $\u \in \mathbb{R}^\N$ with $\A\u=\b$ and $\| \u \|_1 \leq d \sqrt{s_{*}} \| \b \|_2$, where $s_{*}= \M / \ln(e\N/\M)$.
\end{definition}

Note that many types of matrices satisfy the robust NSP of order $s$ and the $\ell_1$-quotient property, e.g., Gaussian random matrices, randomly subsampled Fourier-matrices, and randomly subsampled circulant matrices. For instance, a properly scaled matrix with i.i.d.\ Gaussian entries satisfies both properties with high probability if $\M \ge C s \log(e\N/s)$ and $\M \le \N/2$, where the constant $C>0$ only depends on the NSP parameters $\rho$ and $\tau$ \citep{foucart2013compressed}.
Combining Theorem \ref{theorem:L1_equivalence_positive_PartII} and Theorem~3 of \cite{chou2021more}, the following stable recovery result can be derived.

\begin{theorem}[Stability]\label{theorem:robustness}
Let $\A \in \R^{\N \times \M}$ be a matrix satisfying the $\ell_2$-robust null space property with constants $0 \leq \rho < 1$ and $\tau > 0$ of order $s := c \M/\log(e\N/\M)$ and the $\ell_1$-quotient property with respect to the $\ell_2$-norm with constant $d > 0$.

For $\x_* \in \R^N$ and $\y = \A \x_*$, {recall $\y_+$ and $C_+$ from \eqref{eq:CPlus}}. Decompose $\x_*$ into
\begin{equation}
    \x_* = \x_+ - \x_-
\end{equation}
where $\A\x_+ = \y_+$. For $\epsilon > 0$ assume that $\alpha > 0$ satisfies
\begin{equation*}
\alpha \leq h(\|\x_+\|_1,\epsilon)
\end{equation*}
for $h$ defined in \eqref{eq:alpha_epsilon_bound}. Then the limit $\xprodinfty$ defined in Theorems \ref{theorem:L1_equivalence_positive} and \ref{theorem:L1_equivalence_positive_PartII} yields reconstruction error 
\begin{equation}
    \|\xprodinfty - \x_*\|_2 \leq  \frac{C}{\sqrt{\s}}(2 \sigma_s(\x_+)_{\ell_1} + \epsilon) + \|\x_-\|_2.
\end{equation}
The constants $C,C'>0$ only depend on $\rho, \tau, c, d$.
\end{theorem}

As can be seen from  Theorem \ref{theorem:robustness}, our approach to solving \eqref{eq:NNLS} is stable with respect to negative entries of the ground truth, i.e., the reconstruction error depends on the sparsity of the positive part $\x_+$ of $\x_*$ and the magnitude of the negative part $\x_-$ of $\x_*$. The experiments we perform in Section \ref{sec:NumericsStability} suggest that the established solvers for \eqref{eq:NNLS} are less stable under such perturbations.

\begin{remark} 
\label{rem:WeightedL1}
The very restrictive form of $\x_0 = \alpha \bo$ in Theorem \ref{theorem:L1_equivalence_positive_PartII} is only required to get an implicit $\ell_1$-bias, which is a classical regularizer for sparse recovery \citep{foucart2013compressed}. By changing the initialization, one can also achieve other biases like weighted $\ell_1$-norms. Below we show that, for $L=2$ and any $\w \in (0,1]^\N$ with $\| \w \|_\infty = 1$, the initialization vector $\x_0$ defined as
\begin{align*}
    \x_0 = e^{-\frac{1}{2} (\1 + \theta\w)}
\end{align*}
will yield an approximate $\ell_{\w,1}$-bias in the limit of gradient flow, where $\| \z \|_{\w,1} = \| \z \odot \w \|_1$ and $\theta > 0$ has to be chosen sufficiently large depending on the aimed for accuracy. Weighted $\ell_1$-norms have been used in various applications, e.g., polynomial interpolation or sparse polynomial chaos approximation
\citep{rauhut2016interpolation}, \citep{peng2014weighted}.
\end{remark}

The statement in Remark \ref{rem:WeightedL1} can be deduced from Theorem \ref{thm:WeightedL1} below by noticing that, for $L=2$, the initialization vector $\x_0$ defined as
\begin{align*}
    (\xe_0)_i = e^{-\frac{1}{2} (1 + \theta \we_i)}
\end{align*}
satisfies $\w = \frac{(-1-\log(\xprod(0)))}{\max_{\n\in[\N]}(-1-\log((\xprode(0))_n)}$ (recall that $\xprod(0) = \x_0^{\odot 2}$ in this case). The scaling $\theta > 0$, which has no effect on $\w$, only has to be chosen sufficiently large to guarantee $\x_0 \le h(Q_+,\epsilon)$.

\begin{theorem}
\label{thm:WeightedL1}
    Let $\varepsilon > 0$ and $L = 2$. Under the assumptions of Theorem \ref{theorem:L1_equivalence_positive_PartII} with $\x_0\leq \alpha\leq h(Q_+,\epsilon)$, we get that 
    \begin{align*}
        \|\xprodinfty\|_{\w,1} - \min_{\z \in S_+}\|\z\|_{\w,1}
        \leq \epsilon,
    \end{align*}
    where $\|\z\|_{\w,1} = \|\z\odot\w\|_1$ denotes the weighted $\ell_1$-norm for
    \begin{align*}
        \w = \delta(-\bo - \log(\xprod_0))
    \end{align*}
    and
    \begin{align*}
        \delta = \frac{1}{\max_{\n\in[\N]}(-1-\log((\xprode_0)_n)}.
    \end{align*}
\end{theorem}

\begin{proof}
For $\L=2$, we obtain from \eqref{eq:optimal_x_positive} that
\begin{align*}
    \xprodinfty \in
    &\argmin_{\z \in S_+}D_{\F}(\z,\xprod(0))
    = \langle \z, \log(\z) -\bo - \log(\xprod_0) \rangle,
\end{align*}
where we used the Equation below (2.10) of \cite{chou2021more} together with the fact that $D_F$ based on $F$ in \eqref{eq:Breman_positive} and $D_F$ of \cite{chou2021more} are the same, cf.\ our discussion below \eqref{eq:Breman_positive}. 
Hence
\begin{align*}
    \langle \xprodinfty, \log(\xprodinfty) -\bo - \log(\xprod_0) \rangle
    \leq\langle \z, \log(\z) -\bo - \log(\xprod_0) \rangle,
\end{align*}
which may be re-stated as
\begin{align*}
    \langle \xprodinfty-\z, -\bo - \log(\xprod_0) \rangle
    \leq\langle \z, \log(\z) \rangle - \langle \xprodinfty, \log(\xprodinfty)  \rangle.
\end{align*}
Let $\alpha\in(0,e^{-\frac{1}{2}})$ and assume that $\alpha\geq\x_0 >0$ (so that $\log(\xprod_0)<-1$). Note that $1\geq\w\geq 0$ by assumption. Denote
\begin{equation*}
    \|\z\|_{\w,1} = \|\z\odot\w\|_1
\end{equation*}
to be the weighted $\ell_1$-norm. By non-negativity of $\w,\xprodinfty,\z$, we get that
\begin{align*}
    \|\xprodinfty\|_{\w,1} - \|\z\|_{\w,1}
    \leq\delta(\langle \z, \log(\z) \rangle - \langle \xprodinfty, \log(\xprodinfty)  \rangle).
\end{align*}
Since $\xi^2 \geq \xi\log(\xi) \geq -e^{-1}$ for $\xi\geq 0$,
\begin{align*}
    \|\xprodinfty\|_{\w,1} - \|\z\|_{\w,1}
    &\leq\delta(\|\z\|_2^2 + \N e^{-1})
    \leq\delta(\|\z\|_1^2 + \N e^{-1})
    \leq \epsilon
\end{align*}
because $\alpha\leq h(Q_+,\epsilon)$ and $\delta\leq -\frac{1}{1+2\log(\alpha)}$. Take the minimum overall $\z\in S_+$, and we get our conclusion.
\end{proof}

%%%%%%%%%%%%%%%%%%%%%%%%%%%
\section{DETAILS FOR ACCELERATED GRADIENT METHODS}
\label{sec:AccelerationDiscretization}
%%%%%%%%%%%%%%%%%%%%%%%%%%%

To transform~\eqref{accelerated-gd-continuous} in an optimization method, we perform a change of variables.
In the case where $L = 2$, $q = 1$ and the new variable is $\ZETA = \log(\halfstep)$, so that the $\halfstep$ equation becomes equivalent to $\ZETA' = - \frac{t}{2} \nabla \LossBasic(\xprod(t))$.
Using an implicit step for $\xprod$ and an explicit step for $\ZETA$ yields the following update scheme:
\begin{equation}
    \label{eq:update_ACC}
    \begin{split}
        \xprod_{k+1} & = \frac{\xprod_k + 2/(k+2) e^{\ZETA_k}}{1 + 2/(k+2)}, \\
        \ZETA_{k+1} & = \ZETA_k - \eta \frac{k+2}{2} \nabla \LossBasic(\xprod).
    \end{split}
\end{equation}

Nesterov's acceleration for the dynamics $\x'(t) = - \nabla \LossPoly(\x)$ is an usual one:
\begin{equation}
    \label{eq:update_Nest}
    \begin{split}
        {\bf y}_{k+1}  & = \x_k - \eta \nabla \LossPoly(\x_k), \\
        \x_{k+1} & = {\bf y}_{k+1} + \frac{k}{k+3}({\bf y}_{k+1} - {\bf y}_k). \\
    \end{split}
\end{equation}

%%%%%%%%%%%%%%%%%%%%%%%%%%%
\section{ADDITIONAL NUMERICAL EXPERIMENTS}
\label{sec:AdditionalNumerics}
%%%%%%%%%%%%%%%%%%%%%%%%%%%

In this section, we provide additional empirical evidence for our theoretical claims. In particular, (i.) we show, for Gaussian signals, the stability of our approach with respect to negative entries, as similarly done in Section \ref{sec:NumericsStability}, (ii.) we illustrate the performance of our method in an over-determined large-scale NNLS problem, and (iii.) we compare our method to projected gradient descent (\textbf{PGD}) \citep{polyak2015projected} in terms of the number of iterations, convergence rate, and running time since it is known that \textbf{PGD} converges linearly to the global minimizer \citep{attouch2013convergence, polyak2015projected} and it is a numerically efficient method due to the fast calculation of the projection step.

\subsection{Experimental setup}

The results of this paper were obtained on a desktop computer running Python 3.9 on an Intel Core i5-7500 CPU with 8GB of RAM.

%%%%%%%%%%%%%%%%%%%%%%%%%%%
\subsection{NNLS Stability with Negative Entries}
\label{sec:NumericsStability_gaussian}
%%%%%%%%%%%%%%%%%%%%%%%%%%%

Here, we continue the comparison of Section~\ref{sec:NumericsStability} and analyze the reconstruction error of several NNLS algorithms with respect to the magnitude of the negative component.

\paragraph{Gaussian signals}
Let $\A\in\mathbb{R}^{M\times N}$ be a random Gaussian matrix, where $\M=30$ and $\N=50$.
We pick a $3$-sparse non-negative vector $\x_+ \in \R_{\geq 0}^\N$ at random.
We, furthermore, define a noise vector $\x_- \in \R_{\geq 0}^\N$ that is $0$ on the support of $\x_+$ and has positive Gaussian entries everywhere else.
For $q \in [0,1]$, we scale $\x_+,\x_-$ such that
\begin{equation}\label{eq:q_level_bis}
    \|\x_+\|_2^2 = 1-q
    \quad \text{and} \quad
    \|\x_-\|_2^2 = q.
\end{equation}
The perturbed signal is then given by $\x = \x_+ - \x_-$, and $q$ regulates the negative corruption.
We regard the copy of $\x_+$ scaled to $\ell_2$-norm $(1-q)$ as ground truth and, by abuse of notation, also refer to it as $\x_+$.
The corresponding measurements are given as $\y = \A\x = \A(\x_+ - \x_-)$.

Figures \ref{fig:NoiseComparison1}, \ref{fig:NoiseComparison2} and \ref{fig:NoiseComparison3} compare the reconstruction error $\| \hat\x - \x_+ \|_2$ of \textbf{HL-NNLS}, \textbf{TNT-NN}, \textbf{GD-$3$L}, and \textbf{SGD-$3$L} over various choices of $q$. {The results support our claim that, when applying NNLS to sparse recovery, our approach comes with improved stability against negative perturbations of the ground-truth.}

\begin{figure}[h!]
    \centering
    \begin{subfigure}[b]{0.45\textwidth}
        \centering
        \includegraphics[width = 0.7\textwidth]{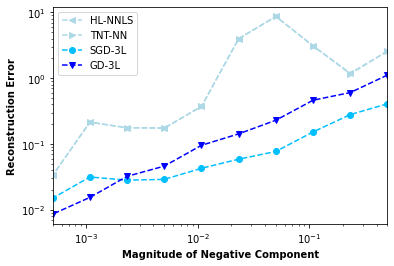}
        \subcaption{Gaussian.}
        \label{fig:NoiseComparison1}
    \end{subfigure} \\
    \begin{subfigure}[b]{0.45\textwidth}
        \centering
        \includegraphics[width = 0.7\textwidth]{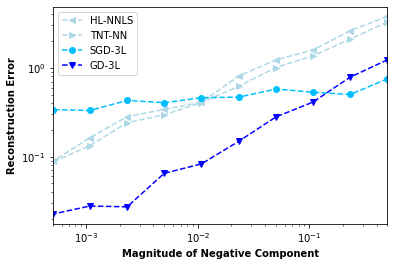}
        \subcaption{MNIST.}
        \label{fig:NoiseComparison2}
    \end{subfigure} \quad 
    \begin{subfigure}[b]{0.45\textwidth}
        \centering
        \includegraphics[width = 0.7\textwidth]{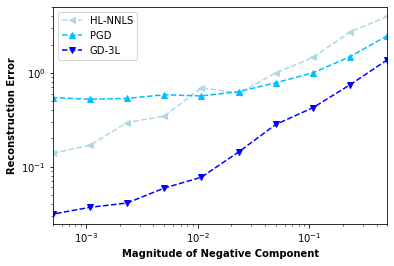}
        \subcaption{MNIST with PGD.}
        \label{fig:NoiseComparison3}
    \end{subfigure}
    \caption{Comparison of stability with respect to negative entries, see Section \ref{sec:NumericsStability_gaussian}.}
\end{figure}

%%%%%%%%%%%%%%%%%%%%%%%%%%%
\subsection{Comparison of GD variants for solving NNLS}
%\label{sec:LargeScaleNNLS}
%%%%%%%%%%%%%%%%%%%%%%%%%%%

In this section, we present experiments comparing various gradient descent (GD) variants for solving non-negative least squares (NNLS) problems. We examine the performance of these methods on both dense and sparse problems, focusing on convergence behavior and support identification. We consider two scenarios:
\begin{enumerate}
\item Dense case: A $512\times512$ linear system with a dense vector with sparsity equal to 512.
\item Sparse case: A $512\times1024$ linear system with an 8-sparse vector $\x \in \mathbb{R}^{512}$.
\end{enumerate}
For both cases, we use a vanilla step size of $3.16\times10^{-1}$. We compare the following methods: Vanilla \textbf{GD-$2$L} with constant stepsize, ACC (our accelerated method based on Theorem \ref{thm:Accelerated}), Nesterov's accelerated gradient descent, Restarted Nesterov's method, \textbf{GD-$2$L} with the Barzilai-Borwein step size Barzilai-Borwein (BB) stepsize 
$$\eta = \frac{\|\x_t-\x_{t-1} \|_2^2}{\|\A(\x_t-\x_{t-1}) \|_2^2},$$ 
\citep[e.g.,][]{barzilai1988two}, and Projected Gradient Descent (\textbf{PGD}), which proceeds as follows.

\fbox{
\begin{minipage}[c]{0.9\linewidth}
    \textbf{PGD:} Initialize with $\x_0 \geq 0$, e.g., $\x_0 = 0$. And for $k = 0, 1, 2, \ldots$ until convergence:
\begin{itemize}
    \item[(i)] Compute the gradient: $\nabla f(\x_k) = \A^T(\A\x_k - \b)$.
    \item[(ii)] Choose step size.
    \item[(iii)] Update: $\x_{k+1} = P_+[\x_k - \alpha_k \nabla f(\x_k)]$, where $P_+$ is the projection onto the non-negative orthant $P_+[\x]_i = \max(0, \x_i)$.
\end{itemize}
\end{minipage}
}

We also examine the effect of different initialization scales, ranging from $\alpha = 10^{-9}$ to $10^{-1}$.

\begin{figure}[ht]
    \centering
    \includegraphics[width=0.9\textwidth]{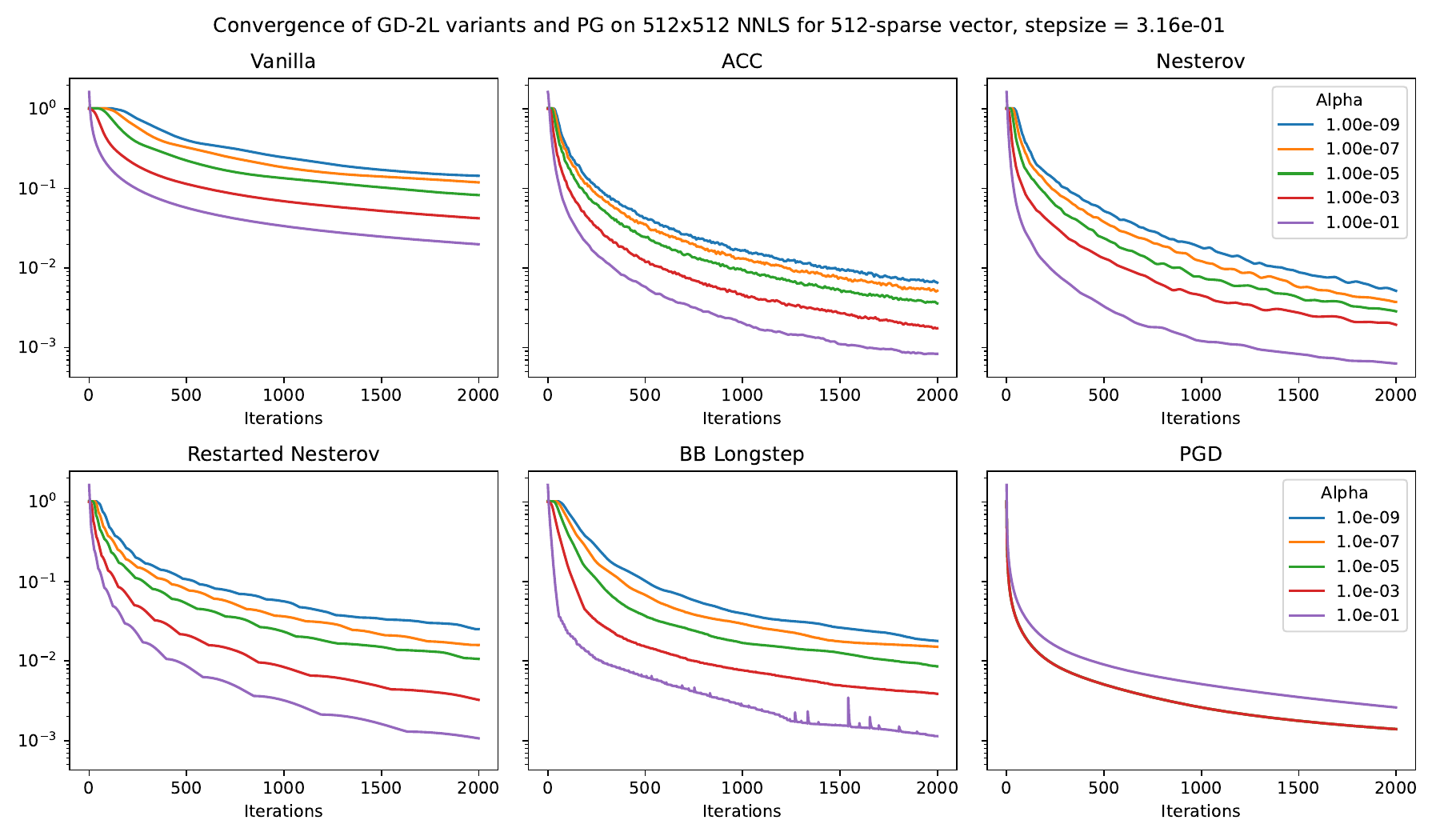}
    \caption{Error $\norm{\A\xprod - \b}{2}$ for GD-2L variants and \textbf{PGD} in the dense case.}
    \label{fig:Convergence_Support_GD2L_PGD_denseCase}
\end{figure}

Figure \ref{fig:Convergence_Support_GD2L_PGD_denseCase} shows the convergence behavior of various \textbf{GD-$2$L} variants and \textbf{PGD} for the dense case. The y-axis represents the error $\|\A\x - \b\|_2$, and the x-axis shows the number of iterations.
We see that all methods decreased the error below $10^{-1}$ provided a sufficient number of iterations. The ACC and Nesterov variants demonstrate improved convergence rates compared to vanilla \textbf{GD-$2$L}, while the BB Longstep method shows competitive performance, especially in later iterations. Notably, the initialization scale ($\alpha$) has a significant impact on convergence speed, with larger $\alpha$ generally leading to faster convergence.

Figure \ref{fig:Convergence_Support_GD2L_PGD} presents two aspects of the sparse case: convergence behavior, cf. Subfigure \ref{fig:Convergence_Support_GD2L_PGD} (a) and support identification, cf. Subfigure \ref{fig:Convergence_Support_GD2L_PGD}(b).
We observe that the overall trends are similar to the dense case but with faster convergence due to the problem's sparsity. \textbf{PGD} again shows rapid initial progress, while the accelerated methods (ACC, Nesterov, Nesterov with a restarting strategy that improves its convergence, e.g., \citep{roulet2017sharpness}) demonstrate improved convergence rates over vanilla \textbf{GD-$2$L}. Unlike the dense case, smaller initialization scales generally lead to faster convergence, especially for the accelerated methods.

Figure \ref{fig:Convergence_Support_GD2L_PGD} illustrates how well each method identifies the true support of the sparse solution over iterations. We find that despite being a fast method, \textbf{PGD} struggles to achieve accurate support identification. In contrast, all variants of the overparametrized GD perform well in this task.

The experiments demonstrate that the variants of overparametrized \textbf{GD-$2$L}, particularly the accelerated methods (ACC, Nesterov, and Restarted Nesterov), offer significant advantages in solving NNLS problems. These methods consistently show improved convergence rates compared to vanilla \textbf{GD-$2$L}, especially in the sparse case. Moreover, they excel in support identification, a crucial aspect in many applications requiring sparse solutions.

The overparametrized GD variants demonstrate remarkable flexibility across different problem settings. They perform well in both dense and sparse scenarios, adapting effectively to various initialization scales. This adaptability, combined with their strong performance in support identification, makes them particularly suitable for a wide range of NNLS applications.

The BB Longstep variant also demonstrates competitive performance, especially in later iterations, offering another valuable option within the overparametrized GD framework. Furthermore, the ACC method, based on our theoretical results in Theorem \ref{thm:Accelerated}, shows promising performance, validating the practical relevance of our theoretical contributions. 

\begin{figure}[ht]
    \centering
    \begin{subfigure}[b]{\textwidth}
        \centering
        \includegraphics[width=0.9\textwidth]{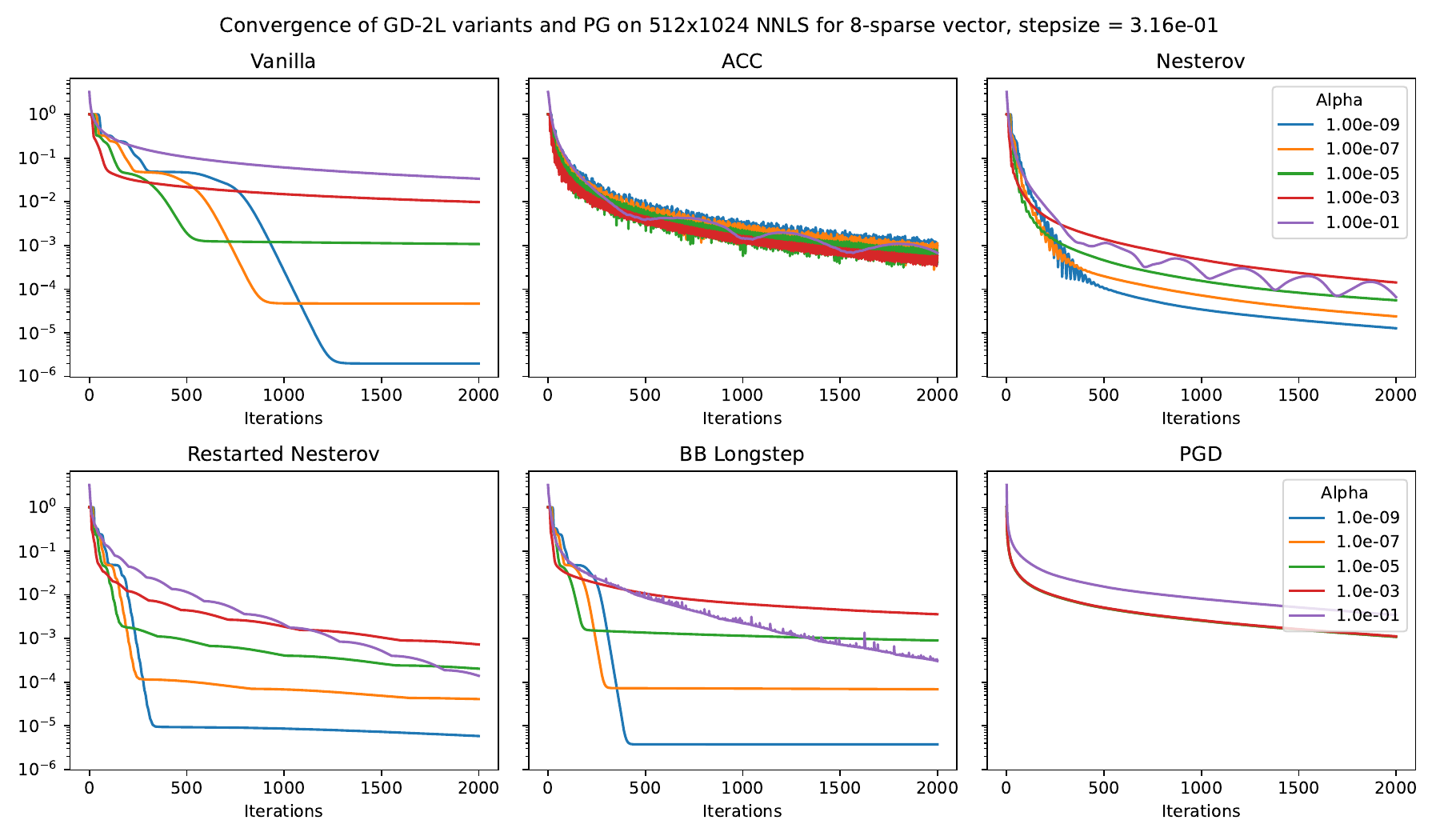}
	    \subcaption{Error $\norm{\A\xprod - \b}{2}$ along iterations.}
    \end{subfigure}

    \begin{subfigure}[b]{\textwidth}
        \centering
        \includegraphics[width=0.9\textwidth]{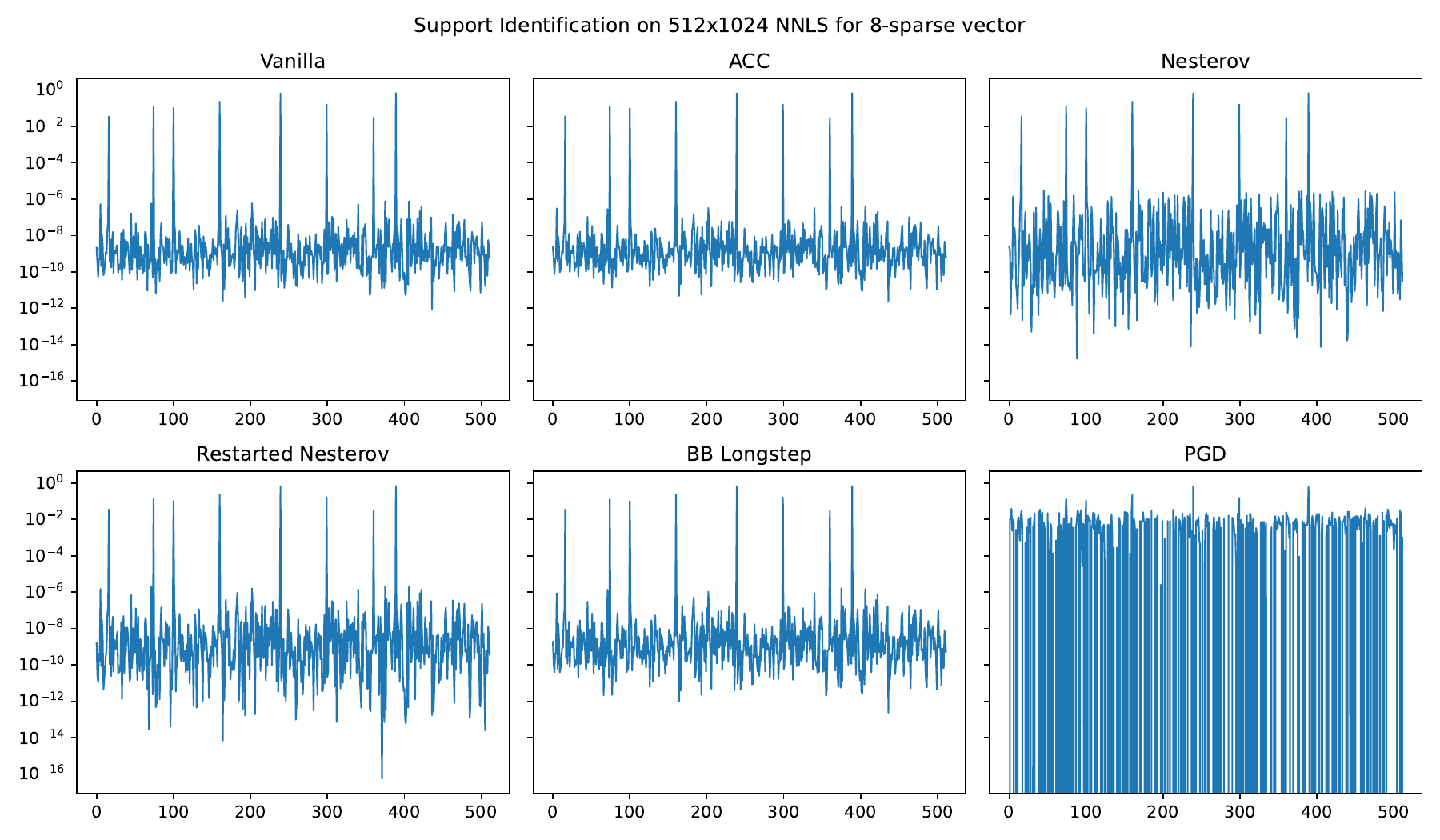}
	    \subcaption{Support identification}
    \end{subfigure}
    \caption{Comparison of \textbf{GD-$2$L} variants and \textbf{PGD} in the sparse case.}
    \label{fig:Convergence_Support_GD2L_PGD}
\end{figure}

%%%%%%%%%%%%%%%%%%%%%%%%%%%
\subsection{Experiments on large linear systems}
\label{sec:LargeScaleNNLS}
%%%%%%%%%%%%%%%%%%%%%%%%%%%

In this experiment, cf.\ Figure \ref{fig:LargeScale}, we illustrate the performance of our method for two large-scale NNLS problems.
The first, in Figure~\ref{fig:2048x1024}, treats an overdetermined setting, where the matrix $\A \in \R^{2048\times 1024}$ is standard Gaussian and, for a Gaussian random vector $\x \in \R^{1024}$, $\y$ is created as a perturbed version of $\A\x$ such that $\y$ is not in the range of $\A$.
We take $\x_0 = \boldsymbol{1}$ as a generic initialization.
Figure \ref{fig:2048x1024} shows the error $\| \A\x^{\odot L} - \y \|_2$ over the iterations of gradient descent.
As predicted by Theorem \ref{theorem:L1_equivalence_positive}, our method converges to a solution that solves the NNLS problem (compare the error to the benchmark given by the Lawson-Hanson algorithm).

\begin{figure}[ht]
    \centering
    \begin{subfigure}[b]{0.45\textwidth}
        \centering
        \includegraphics[width=\textwidth]{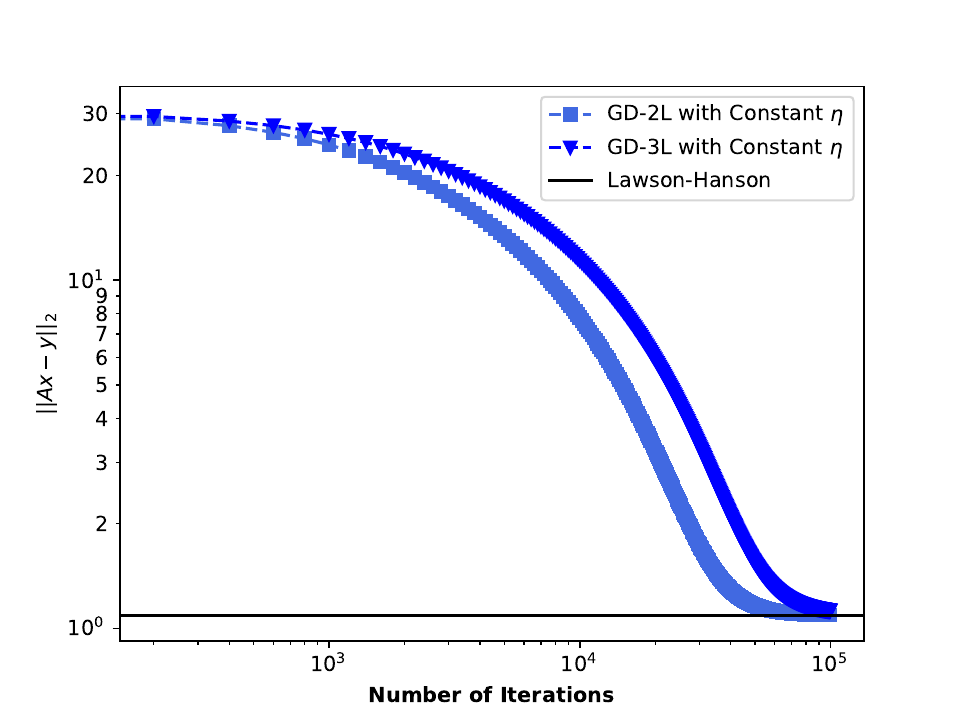}
        \subcaption{Dense setting.}
        \label{fig:2048x1024}
    \end{subfigure}
    \hfill
    \begin{subfigure}[b]{0.45\textwidth}
        \centering
        \includegraphics[width=1\textwidth]{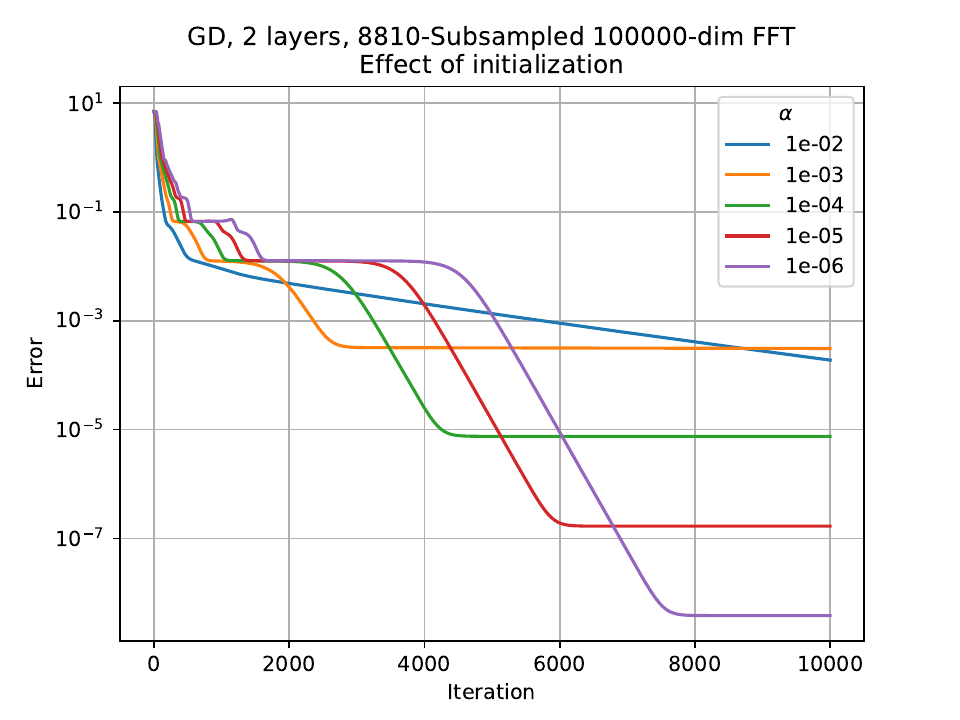}
	    \subcaption{Subsampled FFT setting.}
        \label{fig:LargeFFT}
    \end{subfigure}
    \caption{Large-scale NNLS problems, see Section \ref{sec:LargeScaleNNLS}.}
    \label{fig:LargeScale}
\end{figure}

A second experiment deals with a subsampled FFT, using $8\,810$ random rows of a $100\,000$-dimensional FFT to build the matrix $\A$.
A ground-truth $\x^*$ with $50$ positive entries is chosen, and again we set $\b = \A\x^*$.
We use two layers, set a constant stepsize of $0.2$, and vary the initialization $\x_0 = \alpha \1$ for $\alpha$ from $10^{-6}$ to~$10^{-2}$.
As shown in Figure~\ref{fig:LargeFFT}, we obtain two different error regimes: a smoothly decaying error for the relatively ``large'' initialization $10^{-2}$ and a slower but sharper decay that essentially stagnates after a given point, for the remaining ones.
Notice that the stagnation is smaller as the initialization gets closer to zero, as expected by our theory, and it corresponds to better support identification, while the initialization at $\alpha = 10^{-2}$ leads to a solution that is significantly different from the initial sparse ground truth.

\medskip

Our third experiment analyses the behavior of several algorithms for NNLS over a sequence of square problems with dimensions $500$, $1000$, $2000$ and $4000$.
In addition to our reparameterized GD algorithms, indicated as \textbf{GD} (for 2-layer gradient descent) and its accelerated versions \textbf{GD-BB} (with Barzilai-Borwein large-step), \textbf{GD-Nesterov} (accelerating GD, equation~\eqref{eq:update_Nest}) and \textbf{MD-Nesterov} (as in equation~\eqref{eq:update_ACC}, the discrete version of \eqref{accelerated-gd-continuous}), we include Lawson and Hanson's algorithm~\textbf{LH}, projected gradient descent~\textbf{PG} and accelerated projected gradient descent~\textbf{APG}.
For each matrix size, we selected 10 Gaussian matrices and a target vector with positive entries, and stored running time and final error, $\norm{Ax - b}{2}$ at $10\,000$ iterations.

In Figure~\ref{fig:square_timing}, we notice that the Lawson and Hanson algorithm does not scale as well as the others, essentially due to its need to solve larger and larger linear systems, whereas the others only rely on matrix-vector multiplication. Besides that, we emphasize once more that this algorithm, to our best knowledge, does not have asymptotic convergence rates, but it is only known that it stops after a finite number of iterations.

In Figure~\ref{fig:square_error}, we observe that the error that each method obtains after $10\,000$ iterations is relatively stable as the matrix size increases.
Lawson and Hanson's algorithm reaches an almost exact solution and is therefore not displayed to avoid distorting the axis.  Among the visible methods, \textbf{APG} achieves the lowest error (around $10^{-5}$), showing a significant advantage over other approaches. Our accelerated methods (GD-BB, GD-Nesterov, and MD-Nesterov) form a middle cluster with errors between O($10^{-4}$) and O($10^{-3}$), demonstrating the effectiveness of acceleration techniques in our reparameterized framework. As expected from our previous experiments, vanilla gradient descent converges more slowly, reaching only O($10^{-2}$) error within the same number of iterations.

While projected gradient methods like \textbf{PG} and \textbf{APG} achieve smaller regression errors for the same number of iterations, our experiments demonstrate that they are less effective for sparse recovery tasks. Our reparameterized GD methods excel in identifying the correct support of sparse solutions and show better stability when dealing with negative perturbations in the ground truth signal. Moreover, Figure Figure~\ref{fig:square_timing} confirms that our methods scale efficiently with problem size, making them particularly suitable for large-scale applications where both computational efficiency and support identification are important considerations.

\begin{figure}[ht]
    \centering
    \begin{subfigure}[b]{0.45\textwidth}
        \centering
        \includegraphics[width=1\textwidth]{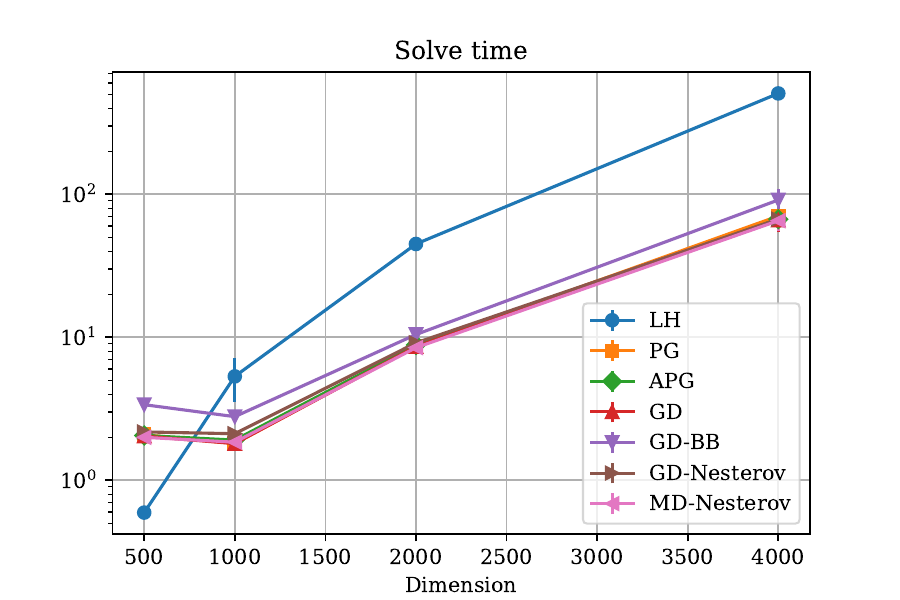}
	    \subcaption{Time to perform 10k iterations}
        \label{fig:square_timing}
    \end{subfigure}
    \hfill
    \begin{subfigure}[b]{0.45\textwidth}
        \centering
        \includegraphics[width=\textwidth]{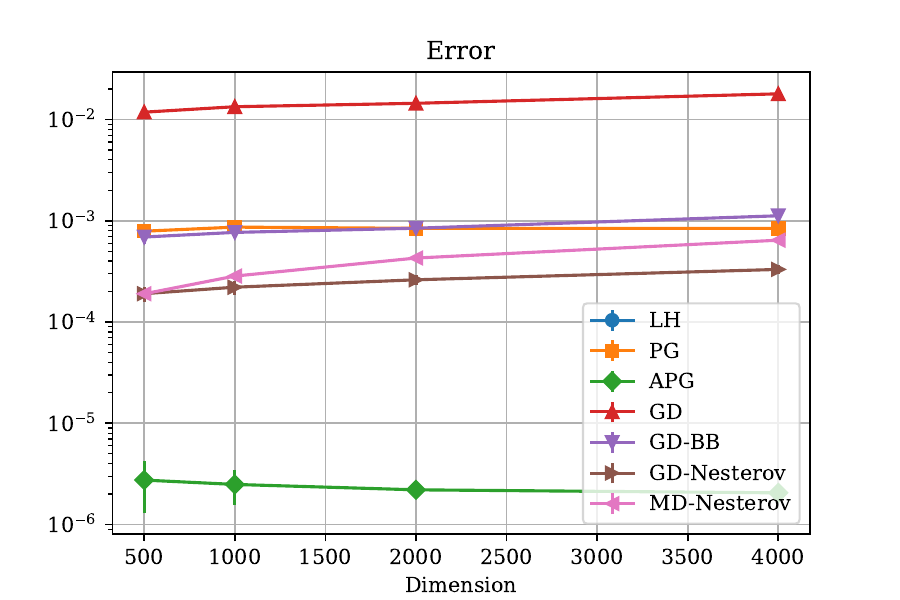}
        \subcaption{Error at 10k iterations}
        \label{fig:square_error}
    \end{subfigure}
    \caption{Timing and error behavior across matrix sizes for square NNLS problems, see Section \ref{sec:LargeScaleNNLS}.}
    \label{fig:square_scale}
\end{figure}

%%%%%%%%%%%%%%%%%%%%%%%%%%%
\subsection{Visual experiments: using NNLS for retrieving images}
%\label{sec:LargeScaleNNLS}
%%%%%%%%%%%%%%%%%%%%%%%%%%%

In Figure \ref{fig:cifar10} we report the result of the CIFAR10 experiment in Section \ref{sec:NumericsStability}. These visualizations demonstrate how different NNLS methods perform when reconstructing grayscale CIFAR10 images corrupted with negative components. The images clearly illustrate the varying abilities of these algorithms to recover visual information under challenging conditions. Unlike the MNIST experiments, we observe that for CIFAR10 images (which typically have less distinct sparse structures), \textbf{GD-$3$L} provides the most accurate reconstructions while \textbf{SGD-$3$L} tends to induce excessive sparsity, particularly at higher corruption levels ($q = 0.5$).

\begin{figure*}[ht]
\captionsetup[subfigure]{justification=centering}
    \centering
    \begin{subfigure}[b]{0.19\textwidth}
        \centering
        \includegraphics[width = .5\textwidth]{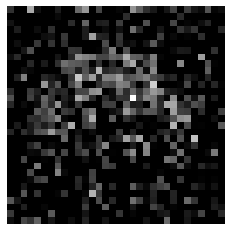}
        \subcaption*{HL-NNLS \\ ($q=0.5$)}
    \end{subfigure}
    \hfill
    \begin{subfigure}[b]{0.19\textwidth}
        \centering
        \includegraphics[width = .5\textwidth]{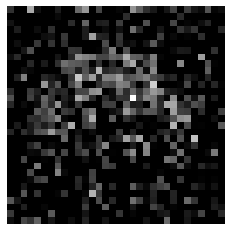}
        \subcaption*{TNT-NN \\ ($q=0.5$)}
    \end{subfigure}
    \hfill
    \begin{subfigure}[b]{0.19\textwidth}
        \centering
        \includegraphics[width = .5\textwidth]{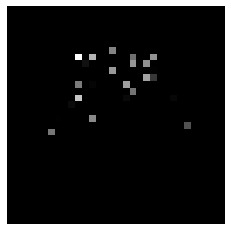}
        \subcaption*{SGD-3L \\ ($q=0.5$)}
    \end{subfigure}
    \hfill
    \begin{subfigure}[b]{0.19\textwidth}
        \centering
        \includegraphics[width = .5\textwidth]{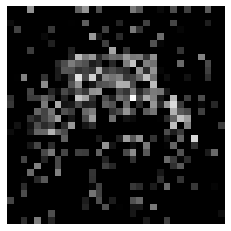}
        \subcaption*{GD-3L \\ ($q=0.5$)}
    \end{subfigure}
    \hfill
    \begin{subfigure}[b]{0.19\textwidth}
        \centering
        \includegraphics[width = .5\textwidth]{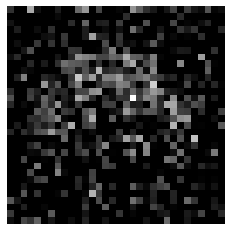}
        \subcaption*{PGD \\ ($q=0.5$)}
    \end{subfigure}\\
    \begin{subfigure}[b]{0.19\textwidth}
        \centering
        \includegraphics[width = .5\textwidth]{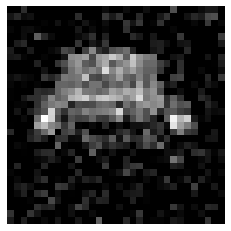}
        \subcaption*{HL-NNLS \\ ($q=0.05$)}
    \end{subfigure}
    \hfill
    \begin{subfigure}[b]{0.19\textwidth}
        \centering
        \includegraphics[width = .5\textwidth]{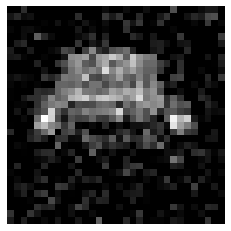}
        \subcaption*{TNT-NN \\ ($q=0.05$)}
    \end{subfigure}
    \hfill
    \begin{subfigure}[b]{0.19\textwidth}
        \centering
        \includegraphics[width = .5\textwidth]{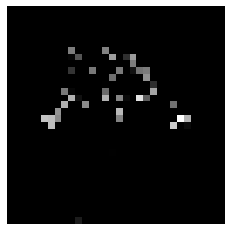}
        \subcaption*{SGD-3L \\ ($q=0.05$)}
    \end{subfigure}
    \hfill
    \begin{subfigure}[b]{0.19\textwidth}
        \centering
        \includegraphics[width = .5\textwidth]{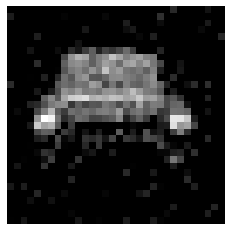}
        \subcaption*{GD-3L \\ ($q=0.05$)}
    \end{subfigure}
    \hfill
    \begin{subfigure}[b]{0.19\textwidth}
        \centering
        \includegraphics[width = .5\textwidth]{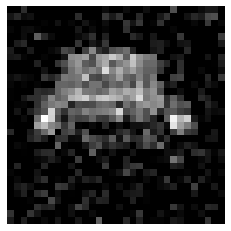}
        \subcaption*{PGD \\ ($q=0.05$)}
    \end{subfigure}
    \caption{Illustration of the CIFAR10 reconstruction, see Section \ref{sec:NumericsStability}.}
    \label{fig:cifar10}
\end{figure*}

\end{document}